\documentclass[a4paper, 11pt]{article}
\usepackage[utf8]{inputenc}
\usepackage{amsfonts}
\usepackage{amsmath}
\usepackage{amsthm}
\usepackage{amssymb}
\usepackage{verbatim}
\usepackage{setspace}
\usepackage{tikz}
\usepackage{slashed}
\usepackage[margin=1in]{geometry}
\usepackage{mathtools}
\usepackage{url}
\usepackage{graphicx}
\usepackage{enumitem}

\usetikzlibrary{matrix}

\newtheorem{thm}{Theorem}[section]
\newtheorem{lemma}[thm]{Lemma}
\newtheorem{prop}[thm]{Proposition}
\newtheorem{defn}[thm]{Definition}

\newtheorem{rmk}[thm]{Remark}

\def\res{\operatorname{res}}

\def\Dom{\operatorname{dom}}

\def\Range{\operatorname{range}}

\def\id{\operatorname{id}}

\def\hotimes{\operatorname{\hat{\otimes}}}

\def\gv{\operatorname{gv}}

\def\rank{\operatorname{rank}}

\def\Cliff{\operatorname{\mathbb{C}liff}}
\def\cliff{\operatorname{Cliff}}

\def\CCl{\operatorname{\mathbb{C}\ell}}
\def\Ad{\operatorname{Ad}}

\def\End{\operatorname{End}}
\def\tr{\operatorname{tr}}
\def\Tr{\operatorname{Tr}}

\def\FF{\operatorname{\mathcal{F}}}
\def\GG{\operatorname{\mathcal{G}}}
\def\BB{\operatorname{\mathcal{B}}}

\def\DD{\operatorname{\mathcal{D}}}
\def\KK{\operatorname{\mathcal{K}}}
\def\LL{\operatorname{\mathcal{L}}}
\def\AA{\operatorname{\mathcal{A}}}
\def\MM{\operatorname{\mathcal{M}}}
\def\DD{\operatorname{\mathcal{D}}}
\def\UU{\operatorname{\mathcal{U}}}

\def\HH{\operatorname{\mathcal{H}}}
\def\NN{\operatorname{\mathcal{N}}}
\def\EB{\operatorname{\mathbb{E}}}
\def\CB{\operatorname{\mathbb{C}}}
\def\RB{\operatorname{\mathbb{R}}}
\def\ZB{\operatorname{\mathbb{Z}}}
\def\NB{\operatorname{\mathbb{N}}}

\title{The Godbillon-Vey invariant and equivariant $KK$-theory}
\author{Lachlan MacDonald, Adam Rennie\\
\\
School of Mathematics and Applied Statistics\\
University of Wollongong\\
Northfields Ave, Wollongong, NSW, 2522}
\date{October 2019}

\begin{document}

\maketitle

\begin{abstract}
We construct a groupoid equivariant Kasparov class
for transversely oriented foliations in all codimensions. In
codimension 1 we show that the Chern character of an
associated semifinite spectral triple recovers the Connes-Moscovici 
cyclic cocycle for the Godbillon-Vey secondary characteristic class.
\end{abstract}

\section{Introduction}

In this paper we construct a semifinite spectral triple 
for codimension 1 foliations whose Chern
character is a global, non-\'{e}tale version of the cyclic cocycle, constructed by Connes and Moscovici \cite{backindgeom},
representing the Godbillon-Vey class. The construction passes through
groupoid equivariant Kasparov theory, and this initial part of the construction
works in all codimensions. 

Associated to any foliated manifold $(M,\FF)$ of 
codimension $q$ is a canonical real rank $q$ 
vector bundle $N = TM/T\FF$ called the normal bundle.  
One of the foundational results of the theory of foliated 
manifolds is \emph{Bott's vanishing theorem}, which 
states that the Pontrjagin classes $p^{i}(N)$ of the normal 
bundle $N$ must vanish for all $i>2q$ \cite{bott1}.  
This vanishing theorem guarantees the existence of 
new characteristic classes for $M$ called 
\emph{secondary characteristic classes}, which 
have been studied extensively \cite{bott2,botthaef,foliatedbundles}.  
It has been shown in particular that all such 
classes arise under the image of a characteristic map from the Gelfand-Fuchs cohomology 
of the Lie algebra of formal vector fields \cite{gelfuks} to the cohomology of $M$ \cite{bott3,botthaef}.

The most famous example of a secondary characteristic 
class is the Godbillon-Vey invariant, first discovered by Godbillon and Vey \cite{gv}, which arises in the 
context of transversely orientable foliations and can be 
constructed explicitly at the level of differential forms.  
More specifically, transverse orientability of a 
codimension $q$ foliated manifold $(M,\FF)$ amounts 
to the existence of a nonvanishing section of the top 
degree line bundle $\Lambda^{q}N^{*}$ of the conormal 
bundle $N^{*}$ over $M$.  Any identification of $N^*$ with 
a subbundle of $T^*M$, obtained say by equipping $M$ with a 
Riemannian metric, identifies such a section with a 
nonvanishing differential form $\omega\in\Omega^{q}(M)$ 
such that 
\begin{equation}
\label{omega}
    \omega(X_{1}\wedge\cdots\wedge X_{q}) = 0
\end{equation}
whenever any one of the $X_{j}$ is contained in the 
space $\Gamma(T\FF)$ of vector fields which are tangent 
to the foliation.  Since the subbundle 
$T\FF\subset TM$ is integrable, by the Frobenius 
theorem one is guaranteed the 
existence of a 1-form $\eta\in\Omega^{1}(M)$ for which
\[
    d\omega = \eta\wedge\omega.
\]
The differential form 
$\eta\wedge(d\eta)^{q}$ is closed, and its class 
$GV$ in de Rham cohomology is independent 
of the choices of $\omega$ and $\eta$.  The Godbillon-Vey invariant has  
been shown to be closely related to measure theory and dynamics: see 
\cite{cantcon,Duminy,hh,hurd} for example.

Building on work of Winkelnkemper \cite{wink} which associated to any foliated manifold $(M,\FF)$ its holonomy groupoid $\GG$, Connes \cite{folops} initiated the study of foliated manifolds as noncommutative geometries using the convolution algebra $C_{c}^{\infty}(\GG)$.  While the convolution algebra $C_{c}^{\infty}(\GG)$ associated to the full holonomy groupoid is necessary when considering leafwise phenomena \cite{cs}, when considering \emph{transverse geometry only} one can simplify matters in the following way.  Choose a $q$-dimensional submanifold $T$ of $M$ which intersects each leaf of $\FF$ at least once and which is everywhere transverse to $\FF$ (such a $T$ can be found by taking the disjoint union of local transversals in $M$ defined by a covering of foliated charts).  Then the restricted groupoid
\[
(\GG)_{T}^{T}:=\{u\in \GG:s(u),r(u)\in T\}
\]
inherits a differential topology from $\GG$ under which it is an \'{e}tale Lie groupoid \cite[Lemma 2]{crainic2}.  Importantly, the groupoids $(\GG)_{T}^{T}$ and $\GG$ are \emph{Morita equivalent} \cite[Lemma 2]{crainic2}.  They therefore have the same cyclic (co)homology \cite{crainic2, crainic1}, and have Morita equivalent $C^{*}$-algebras \cite{muhwil} so are the same as far as $K$-theory is concerned also.  In treating the transverse geometry of a foliation it has therefore become standard in the literature to use the \'{e}tale groupoid $(\GG)_{T}^{T}$ in the place of the full holonomy groupoid $\GG$ \cite{cyctrans,hopf1,goro1,diffcyc,goro2,crainic1,moscorangi, backindgeom}.

A reasonable model for any such \'{e}tale groupoid $(\GG)^{T}_{T}$ is simply the action groupoid $V\rtimes\Gamma$, where $V$ is an oriented manifold (a stand-in for the transversal $T$), and where $\Gamma$ is a discrete group of orientation-preserving diffeomorphisms of $V$ (which is a stand-in for the action of the holonomy groupoid on $T$).  It is in this setting that Connes shows \cite[Theorem 7.15]{cyctrans} that all Gelfand-Fuchs cohomology classes (hence all secondary characteristic classes) can be represented by cyclic cocycles on $C_{c}^{\infty}(V)\rtimes\Gamma$.  Connes gives in particular an explicit formula for the 
cyclic cocycle defined by the Godbillon-Vey invariant 
when $V = S^{1}$.  If $dx$ denotes the standard volume form on $S^{1}$, then associated to any $g\in\Gamma$ is the $\RB$-valued group cocycle
\[
\ell(g):=\log\bigg(\frac{d(x\cdot g^{-1})}{dx}\bigg).
\]
Connes shows that the formula
\begin{equation}\label{cgv}
    \phi_{GV}(f^{0},f^{1},f^{2})
    :=\sum_{g_{0}g_{1}g_{2}=1_{\Gamma}}\int_{S^{1}}
    f^{0}(x)f^{1}(x\cdot g_{0})f^{2}(x\cdot g_{0}g_{1})
    \big(d\ell(g_{1}g_{2})\ell(g_{2}) - \ell(g_{1}g_{2})d\ell(g_{2})\big)
\end{equation}
defines a cyclic 2-cocycle on $C_{c}^{\infty}(V)\rtimes\Gamma$, 
and that the class of this 2-cocycle coincides with that 
defined by the Godbillon-Vey invariant.

More recently, Connes and Moscovici have used a deep link with Hopf symmetry \cite{hopf1} to construct a characteristic map sending Gelfand-Fuchs cocycles to cyclic cocycles on the convolution algebra $C_{c}^{\infty}(F^{+}(V))\rtimes\Gamma$ of the groupoid $F^{+}(V)\rtimes\Gamma$ associated to the lift of $\Gamma$ to the oriented frame bundle $F^{+}(V)$ for $V$.  Connes and Moscovici show in \cite{backindgeom} that the formula
\begin{equation}\label{cmgv}
	\tilde{\phi}_{GV}(a^{0},a^{1}):=\sum_{g_{0}g_{1} = 1_{\Gamma}}\int_{F^{+}(V)}a^{0}(y)(\delta_{1}a^{1})(y\cdot g_{0})\tilde{\omega}(y),
\end{equation}
where $\delta_{1}$ is a derivation of $C_{c}^{\infty}(F^{+}(V))\rtimes\Gamma$ related to $d\ell$ and where $\tilde{\omega}$ is a $G$-invariant transverse volume form on $F^{+}(V)$, defines a 1-cocycle on $C_{c}^{\infty}(F^{+}(V))\rtimes\Gamma$ that represents the Godbillon-Vey invariant.  As will be shown in this paper, the derivation $\delta_{1}$ can be realised in the \emph{non-\'{e}tale setting} of the \emph{full holonomy groupoid $\GG$} of a foliated manifold, where it arises as a commutator between convolution along $\GG$ with a dual Dirac operator on a Hilbert space of sections of an exterior algebra bundle. In noncommutative geometry, the Godbillon-Vey invariant has since been further explored in groupoid cohomology \cite{crainic1}, cyclic cohomology \cite{goro1,goro2}, via its pairing with the indices of longitudinal Dirac operators \cite{mori1}, and in relation to manifolds with boundary \cite{mori2}.

Accompanying his introduction of the formula \eqref{cgv} for the cyclic cocycle $\phi_{GV}$, 
Connes remarks \cite[Page 4]{cyctrans} that the 
pairing of $\phi_{GV}$ with $K$-theory will not in general 
be integer-valued, which implies that $\phi_{GV}$ must not 
arise as the Chern character of a spectral triple on 
$C_{c}^{\infty}(V)\rtimes\Gamma$.  
Such constraints do not apply to
\emph{semifinite spectral triples}, whose pairings with 
$K$-theory need not lie in the integers, \cite{CC,benfack,cp1}.  

In this paper 
we will recover the analogue of formula \eqref{cmgv} in the global setting of the full holonomy groupoid $\GG$, from a semifinite spectral triple.  Bearing in mind the close relationship between semifinite spectral triples and $KK$-theory \cite{KNR}, this fact can be seen already in the \'{e}tale case of an action groupoid of the form $V\rtimes\Gamma$ using the formalism of differential forms on jet bundles arising from Gelfand-Fuchs cohomology \cite[Proposition 19]{backindgeom}.  An entirely novel nuance of our constructions, however, is that they are \emph{global in nature}, applying immediately to foliated manifolds \emph{without} needing to choose a complete transversal and pass through a Morita equivalence.  This has the advantage of producing cocycles that are defined in terms of global geometric data for $(M,\FF)$, which previously has not been attempted.

We now outline the layout of the paper.  Section 1 will 
discuss the background required on Clifford bundles, groupoid actions, semifinite spectral 
triples and groupoid equivariant $KK$-theory.  Section 2 
will detail the constructions of the $KK$-classes required.  
The constructions of this section are very natural for foliations 
of arbitrary codimension, so will be carried out at this level of 
generality.  Section 3 will consist of the proof of an index 
theorem in codimension 1 which states that the pairing with $K$-theory of
the semifinite spectral triple obtained using the constructions of Section 2 coincides with the pairing coming from the Connes-Moscovici
Godbillon-Vey cyclic cocycle.  Finally in Section 5 we describe how, in codimension 1, our constructions can be viewed as a global geometric analogue of the jet bundle approach described by Connes \cite{cyctrans}.  In particular this justifies our claim that the index formula we obtain really does represent the Godbillon-Vey invariant.

We remark that while the 
spectral triple itself can be easily constructed for foliations 
of arbitrary codimension, it is at this stage unclear whether 
the corresponding index pairing continues to compute the 
pairing of the higher codimension Godbillon-Vey invariant 
with $K$-theory.  We leave this question to future work.

\subsection{Acknowledgements}

LM thanks the Australian Federal Government for a Research Training Program scholarship. AR was partially supported by the BFS/TFS project Pure Mathematics in Norway.  LM and AR 
thank Moulay Benameur for supporting a visit of LM to 
Montpellier in the (northern) Fall of 2018. Both authors 
thank Alan Carey, Bram Mesland, Moulay Benameur, Mathai Varghese and Ryszard Nest 
for helpful discussions, and the anonymous referee for a thoughtful reading. Both authors acknowledge the support of
the Erwin Schr\"{o}dinger Institute where part of this work was conducted.

\section{Background}

Here we recall some basic facts about groupoid actions on spaces, Clifford algebras, semifinite
spectral triples, groupoid actions on algebras and the resulting
equivariant Kasparov theory.

%
We will 
assume that the reader is familiar with locally compact groupoids 
and their associated convolution algebras \cite{folops,renault}.
All Hilbert spaces are assumed to be separable.
For such a Hilbert space $\HH$, 
we denote by $\BB(\HH)$ the bounded operators on 
$\HH$ and by $\KK(\HH)$ the compact operators on 
$\HH$.  Inner products on Hilbert modules and Hilbert spaces are assumed to be conjugate-linear in the left variable and linear in the right.

If $X$, $Y$ and $Z$ are sets with maps $f:Y\rightarrow X$ and $g:Z\rightarrow X$, we denote by $Y\times_{f,g}Z$ the fibred product $
\{(y,z)\in Y\times Z:f(y) = g(z)\}$ of $Y$ and $Z$.

\subsection{Clifford algebras}

For our constructions we will need some facts regarding 
Clifford algebras and their representations on exterior 
algebra bundles.  First, if $(V,\langle\cdot,\cdot\rangle)$ 
is a real inner product space with nondegenerate inner product, 
we denote by $\Cliff(V)$ the \emph{complex Clifford algebra} 
of $V$, which is the complexification of the real Clifford algebra
$\cliff(V,\langle\cdot,\cdot\rangle)$. 

There exists a linear isomorphism $\psi_{V}:\Lambda^{*}V\rightarrow\cliff(V,\langle\cdot,\cdot\rangle)$ between the exterior algebra and the Clifford algebra of $V$ defined with respect to any orthonormal basis $\{e_{1},\dots,e_{\rank(V)}\}$ by
\[
\psi_{V}(e_{i_{1}}\wedge\dots\wedge e_{i_{r}}):=e_{i_{1}}\cdot\dots\cdot e_{i_{r}}
\]
for any multi-index $(i_{1},\dots,i_{r})$ with $r\leq\rank(V)$. The isomorphism $\psi_{V}$ determines the structure of a Clifford bimodule on $\Lambda^{*}(V)$, with left action given by
\[
c_{L}(a)w:=\psi_{V}^{-1}(a\cdot\psi_{V}(w))
\]
and right action given by
\[
c_{R}(a)w:=\psi_{V}^{-1}(\psi_{V}(w)\cdot a)
\]
for $a\in\cliff(V)$ and $w\in\Lambda^{*}(V)$.  
We have the following important lemma describing how these representations behave with respect to orthogonal maps.

\begin{lemma}\label{cliffacthom}
	Let $V$ and $W$ be finite dimensional inner product spaces and let $\psi_{V}:\Lambda^{*}V\rightarrow\cliff(V)$, $\psi_{W}:\Lambda^{*}W\rightarrow\cliff(W)$ be the corresponding linear isomorphisms.  Then if $A:V\rightarrow W$ is an orthogonal transformation with induced algebra isomorphisms $A_{\Lambda}:\Lambda^{*}V\rightarrow\Lambda^{*}W$ and $A_{\cliff}:\cliff(V)\rightarrow\cliff(W)$, we have
	\[
	A_{\cliff}\circ\psi_{V} = \psi_{W}\circ A_{\Lambda}.
	\]
\end{lemma}

\begin{proof}
	Regard $V$ as a subspace of $\Lambda^{*}V$ in the usual way, let $\iota:V\rightarrow \cliff(V)$ denote the inclusion map, and consider the map $j:=(\psi_{W}\circ A_{\Lambda})|_{V}:V\rightarrow\cliff(W)$.  Since $A$ is orthogonal, we have $j(v)^2 = \|v\|^21_{\cliff(W)}$ and so by the universal property of the Clifford algebra, there is a unique algebra isomorphism $\phi:\cliff(V)\rightarrow\cliff(W)$ such that $\phi\circ \iota = j$.  Given any vector $v\in V$ we see that
	\[
	j(v) = A_{\cliff}\circ\iota(v)
	\]
	so that $\phi = A_{\cliff}$.  Given an orthonormal basis $\{e_{1},\dots,e_{\dim(V)}\}$ for $V$, and a multi-index $(i_1,\dots,i_k)$ we calculate
	\begin{align*}
	A_{\cliff}\circ\psi_{V}(e_{i_{1}}\wedge\cdots\wedge e_{i_{k}}) &= A_{\cliff}(\iota(e_{i_{1}})\cdots\iota(e_{i_{k}}))\\ &= A_{\cliff}(\iota(e_{i_{1}}))\cdots A_{\cliff}(\iota(e_{i_{k}}))\\ &= \psi_{W}(A_{\Lambda}(e_{i_{1}}))\wedge\cdots\wedge\psi_{W}(A_{\Lambda}(e_{i_{k}}))\\ &= \psi_{W}\circ A_{\Lambda}(e_{i_{1}}\wedge\cdots\wedge e_{i_{k}}),
	\end{align*}
	where the first line is due to the equality $\psi_{V}|_{V} = \iota$, and the second is since $A_{\cliff}$ is an algebra homomorphism.  By linearity we obtain the required identity.
\end{proof}

By abuse of notation, we have a linear isomorphism 
$\psi_{V}:\Lambda^{*}(V)\otimes\CB\rightarrow \Cliff(V)$, which gives, 
by the same formulae as in the real case, 
commuting actions $c_{L}$ and $c_{R}$ of $\Cliff(V)$ 
on $\Lambda^{*}(V)\otimes\CB$. Any orthogonal map 
$A:V\rightarrow W$ of inner product spaces has the property 
that the induced maps $A_{\Cliff}:\Cliff(V)\rightarrow\Cliff(W)$  and
$A_{\Lambda_{\CB}}:\Lambda^{*}(V)\otimes\CB\rightarrow \Lambda^{*}(W)
\otimes\CB$ satisfy $A_{\Cliff}\circ\psi_{V} = \psi_{W}\circ A_{\Lambda_{\CB}}$.

%

%

If $Y$ is a manifold and $E\rightarrow Y$ is a Euclidean 
vector bundle, we obtain a corresponding Clifford algebra 
bundle $\cliff(E)$ and exterior bundle $\Lambda^{*}(E)$, 
as well as corresponding complexifications 
$\Cliff(E) = \cliff(E)\otimes\CB$ and $\Lambda^{*}(E)\otimes\CB$.  
Operating pointwise, we have an isomorphism 
$\psi_{E}:\Lambda^{*}(E)\otimes\CB\rightarrow\Cliff(E)$ 
of vector spaces giving $\Lambda^{*}(E)\otimes\CB$ the 
structure of a $\Cliff(E)$-bimodule, with left and right actions 
denoted, again by abuse of notation, by $c_{L}$ and $c_{R}$ 
respectively.  We will denote by $\CCl(E)$ 
the continuous sections vanishing at infinity of the bundle 
$\Cliff(E)$ over $Y$.  This $\CCl(E)$ is a $C^{*}$-algebra 
and is $\ZB_{2}$-graded by even and odd elements.

\subsection{$\GG$-spaces and $\GG$-bundles}

Let $\GG$ be a groupoid, with unit space $X$ and range and source maps $r:\GG\rightarrow X$ and $s:\GG\rightarrow X$ respectively.  
We say that $\GG$ \emph{acts on (the left of) a set $Y$} or that \emph{$Y$ is a $\GG$-space} if there exists a map $a:Y\rightarrow X$ called the \emph{anchor map} and a map $m:\GG\times_{s,a}Y\rightarrow Y$, denoted $m(u,y):=u\cdot y$, such that
\begin{enumerate}
	\item $a(u\cdot y) = r(u)$ for all $(u,y)\in \GG\times_{s,a}Y$,
	\item $(uv)\cdot y = u\cdot(v\cdot y)$ for all $(v,y)\in \GG\times_{s,a}Y$ and $(u,v)\in \GG^{(2)}$,
	\item $a(y)\cdot y = y$ for all $y\in Y$.
\end{enumerate}
If $\GG$ and $Y$ are topological (resp. smooth) spaces we require the maps $a$ and $m$ to be continuous (resp. smooth).  The simplest example of a $\GG$-space is the unit space $X$ of $\GG$.

If $\GG$ acts on $Y$, we denote by $Y\rtimes \GG$ the space $Y\times_{a,r}\GG$, regarded as a groupoid whose unit space is $Y$, with range and source maps $r(y,u):=y$ and $s(y,u):= u^{-1}\cdot y$ respectively, and with multiplication defined by
\[
	(y,u)\cdot(u^{-1}\cdot y,v):=(y,uv)
\]
for all $(y,u)\in Y\times_{a,r}\GG$ and $(u,v)\in \GG^{(2)}$.  If $\GG$ and $Y$ are topological (resp.) smooth spaces, the groupoid $Y\rtimes \GG$ is equipped with a topological (resp. smooth) structure from its containment as a subspace of the topological (resp. smooth) space $Y\times \GG$.  While for left $\GG$-spaces it is more natural to consider the analogous (and isomorphic) groupoid $\GG\ltimes Y$ obtained from the set $\GG\times_{s,a}Y$, it will be easier for our purposes to use $Y\rtimes \GG$ because, as we will see, our convention in using $\GG$-equivariant Kasparov theory consists in forming pullbacks using the range map rather than the source.

We say that a vector bundle $\pi:E\rightarrow X$ is \emph{$\GG$-equivariant} if $E$ is a $\GG$-space, with $\GG$-action conventionally denoted $(u,e)\mapsto u_{*}e$ and with anchor map $\pi$, and if for each $u\in \GG$ the map $(u,e)\mapsto u_{*}e$ defined on $E_{s(u)}:=\pi^{-1}\{s(u)\}$ is a vector space isomorphism $E_{s(u)}\rightarrow E_{r(u)}$.  More generally, if $\pi:E\rightarrow Y$ is a vector bundle over a $\GG$-space $Y$, we say that $E$ is \emph{$\GG$-equivariant} if it is $Y\rtimes \GG$-equivariant as a bundle over $Y$, in which case we will often denote the map $(Y\rtimes \GG)\times_{s,\pi}E\rightarrow E$, $((y,u),e)\mapsto (y,u)_{*} e$, by simply $(u,e)\mapsto u_{*} e$.  If $\pi:E\rightarrow X$ admits a Euclidean (resp. Hermitian) structure, we say that $E$ is a \emph{$\GG$-equivariant Euclidean (resp. Hermitian) bundle} if for all $(y,u)\in Y\rtimes \GG$ the linear isomorphism $E_{u^{-1}\cdot y}\rightarrow E_{y}$ defined by $(u,e)\mapsto u_{*}e$ is orthogonal (resp. unitary).

If $\pi:E\rightarrow Y$ is 
a $\GG$-equivariant vector bundle over $Y$, then by functoriality $\Lambda^{*}(E)\otimes\CB$ is also an equivariant bundle over $Y$, with action of $u\in \GG$ denoted by $u_{*}:\Lambda^{*}(E|_{Y_{s(u)}})\otimes\CB\rightarrow \Lambda^{*}(E|_{Y_{r(u)}})\otimes\CB$.  If moreover $E$ is an equivariant Euclidean bundle, then by functoriality $\Cliff(E)$ is also an equivariant bundle, with action of $u\in \GG$ denoted by $u_{\diamond}:\Cliff(E|_{Y_{s(u)}})\rightarrow\Cliff(E|_{Y_{r(u)}})$.  In this case, by Lemma \ref{cliffacthom} we have
\begin{equation}
\label{cL}
u_{*}(c_{L}(a)e) = c_{L}(u_{\diamond}a)(u_{*}e)
\end{equation}
and
\begin{equation}
\label{cR}
u_{*}(c_{R}(a)e) = c_{R}(u_{\diamond}a)(u_{*} e)
\end{equation}
for all $u\in \GG$, $a\in \Cliff(E|_{Y_{s(u)}})$ 
and $e\in \Lambda^{*}(E|_{Y_{s(u)}})$.\

When $(M,\FF)$ is a foliated manifold with holonomy groupoid $\GG$, the normal bundle $N = TM/T\FF\rightarrow M$ is a $\GG$-equivariant bundle.  As this fact is fundamental for our constructions, let us briefly review why it is the case.  We choose a countable covering of $M$ by foliated charts $\phi_{i}:U_{i}\cong T_{i}\times P_{i}$, where $T_{i}\subset\RB^{q}$ and $P_{i}\subset \RB^{p}$ are open balls, with change-of-chart maps $\varphi_{j,i}:=\phi_{j}\circ\phi_{i}^{-1}:\phi_{i}(U_{i}\cap U_{j})\rightarrow\phi_{j}(U_{i}\cap U_{j})$ of the form
\[
	\varphi_{j,i}(t,p) = (h_{j,i}(t),\tilde{\varphi}_{j,i}(t,p)),
\]
such that the $h_{i,j}$ are \emph{compatible} in the sense that they satisfy
\[
	h_{i,k} = h_{i,j}\circ h_{j,k}
\]
whenever $U_{i}\cap U_{j}\cap U_{k}\neq\emptyset$.  That such a covering can be chosen can be regarded as the definition of the foliation $\FF$ on $M$ \cite[Chapter 1.2]{cc1}.  We say that a path $\gamma:[0,1]\rightarrow M$ is \emph{leafwise} if its image is entirely contained in a leaf $L$ of $M$, and we refer to its endpoints $\gamma(0)$ and $\gamma(1)$ as its \emph{source} and \emph{range}, denoted $s(\gamma)$ and $r(\gamma)$ respectively. Any leafwise path $\gamma$ whose image is contained in a union $U_{0}\cup U_{1}$ of charts such that $U_{0}\cap U_{1}\neq\emptyset$, and with $s(\gamma)\in U_{0}$ and $r(\gamma)\in U_{1}$, determines a local diffeomorphism $h_{\gamma}:=h_{1,0}$ on a small neighbourhood of $T_{0}\subset\RB^{q}$.  More generally, if the image of a leafwise path $\gamma$ is covered by a chain of charts $\{U_{0},\dots, U_{k}\}$ such that for each $0\leq j<k$ we have $U_{j}\cap U_{j+1}\neq \emptyset$, on a sufficiently small neighbourhood of $T_{0}$ we may define a local diffeomorphism
\[
	h_{\gamma}:=h_{k,k-1}\circ h_{k-1,k-2}\circ\cdots\circ h_{1,0}
\]
mapping onto a small neighbourhood of $T_{k}$.  Because of the compatibility of the $h_{i,j}$, the germ of $h_{\gamma}$ at $s(\gamma)$ does not depend on the chain of charts chosen in its definition.  By definition, the holonomy groupoid $\GG$ consists of equivalence classes of leafwise paths $\gamma$ for which $\gamma_{1}\sim\gamma_{2}$ if and only if $\gamma_{1}$ and $\gamma_{2}$ have the same source and range and the germ at $s(\gamma_{1}) = s(\gamma_{2})$ of $h_{\gamma_{1}}$ is equal to that of $h_{\gamma_{2}}$.

In the coordinates defined by a chart $U_{j}$, the fibres of $N$ identify with tangent vectors to the transversal neighbourhood $T_{j}$, and via this identification it follows that for any leafwise path $\gamma$ in $M$, the derivative of $h_{\gamma}$ furnishes a linear isomorphism
\[
	dh_{\gamma}:N_{s(\gamma)}\rightarrow N_{r(\gamma)}.
\]
It can be seen from the definition of $h_{\gamma}$ that $dh_{\gamma_{1}}\circ dh_{\gamma_{2}} = dh_{\gamma_{1}\gamma_{2}}$ whenever the range of $\gamma_{2}$ is equal to the source of $\gamma_{1}$, where $\gamma_{1}\gamma_{2}$ is the path obtained by concatenating $\gamma_{1}$ and $\gamma_{2}$.  Since local diffeomorphisms with the same germ at a point have the same derivative at that point, to any $u\in \GG$ corresponds a well-defined linear isomorphism $u_{*}:=h_{\gamma}:N_{s(u)}\rightarrow N_{s(u)}$ for any path $\gamma$ that represents $u$.  Since $dh_{\gamma_{1}\gamma_{2}}= dh_{\gamma_{1}}\circ dh_{\gamma_{2}}$, we have $(uv)_{*}= u_{*}\circ v_{*}$ for all $(u,v)\in \GG^{(2)}$, and so $N$ is indeed a $\GG$-equivariant bundle over $M$.

We remark that in general the normal bundle $N$ of a foliated manifold $(M,\FF)$ will not admit the structure of a $\GG$-equivariant Euclidean bundle.  Indeed, the existence of a $\GG$-equivariant Euclidean structure for $N$ implies the existence of a $\GG$-invariant transverse volume form $\omega$ for $(M,\FF)$, and hence implies the existence of a faithful normal semifinite trace on the von Neumann algebra of $(M,\FF)$ defined by restricting functions in the weakly dense algebra $C_{c}(\GG)$ to $M$, and then integrating with respect to $\omega$.  If the Godbillon-Vey invariant of $(M,\FF)$ is nonzero, however, then by results of Hurder and Katok \cite[Theorem 2]{hurdkat} and, in codimension 1, Connes \cite[Theorem 7.14]{cyctrans}, the von Neumann algebra of $(M,\FF)$ contains a type III factor and so admits no nonzero semifinite normal traces.  Examples of foliated manifolds with nonzero Godbillon-Vey invariant are known to be plentiful \cite{thurston1}.

\subsection{Equivariant $KK$-theory for locally Hausdorff groupoids}

Equivariant $KK$-theory for Hausdorff topological groupoids 
was first developed by Le Gall \cite{legall}.  Since foliated manifolds 
generally have only locally Hausdorff holonomy groupoids, Le Gall's 
treatment requires extension for applications to foliation theory.  
Androulidakis and Skandalis \cite{iakovos} have developed an 
equivariant $KK$-theory for the holonomy groupoids arising from 
singular foliations, whose topologies are generally even worse than 
the locally Hausdorff topologies on the holonomy groupoids of regular foliations, 
and which include all regular foliation groupoids as a subclass.  

This section will summarise the required results and definitions 
of Androulidakis and Skandalis in the setting of locally 
Hausdorff Lie groupoids, as well as giving the unbounded 
picture in parallel with work of Pierrot \cite{pierrot}.  See also 
Muhly and Williams \cite{muhwil} and Tu \cite{tu} for useful 
perspectives on non-Hausdorff groupoid actions which have 
further informed the exposition.  

Let $\GG$ be a locally Hausdorff Lie groupoid with locally 
compact Hausdorff unit space $X$, and let $\{U_{i}\}_{i\in I}$ be
a countable cover of $\GG$ by Hausdorff open sets. For each $i\in I$ 
we let $r_{i}:=r|_{U_{i}}$ and $s_{i}:=s|_{U_{i}}$ be the restrictions 
of range and source respectively to the set $U_{i}$.

\begin{defn}
	A \textbf{$C_{0}(X)$-algebra} is a $C^{*}$-algebra $A$ 
	together with a homomorphism $\theta:C_{0}(X)\rightarrow \MM(A)$ 
	into the multiplier algebra of $A$ such that $\theta(C_{0}(X))A = A$.  
	For $a\in A$ and $f\in C_{0}(X)$, we will often denote $\theta(f)a$ 
	by $f\cdot a$.
	
	For $x\in X$, the \textbf{fibre over $x$} is the 
	algebra $A_{x}:=A/I_{x}A$, where $I_{x}$ is the 
	kernel of the evaluation functional $C_{0}(X)\ni f\mapsto f(x)$ on $C_{0}(X)$.
	
	If $A$ and $B$ are $C_{0}(X)$-algebras, 
	a homomorphism $\phi:A\rightarrow B$ is said 
	to be a \textbf{$C_{0}(X)$-homomorphism} 
	if $\phi(f\cdot a) = f\cdot\phi(a)$ for all $f\in C_{0}(X)$ 
	and $a\in A$.  Such a homomorphism induces a family 
	$\phi_{x}:A_{x}\rightarrow B_{x}$ of homomorphisms between the fibres.
\end{defn}

The simplest nontrivial example of a $C_{0}(X)$-algebra is $C_{0}(Y)$, 
where $Y$ is a locally compact Hausdorff space equipped with a 
continuous map $p:Y\rightarrow X$.  The $C_{0}(X)$-structure of $C_{0}(Y)$ is given by $\theta(f)g(y):=f(p(y))g(y)$ for all $f\in C_{0}(X)$ and $g\in C_{0}(Y)$, and the fibre over $x\in X$ 
is $C_{0}(Y)_{x} = C_{0}(Y_{x})$, where $Y_{x}:=p^{-1}\{x\}$.

\begin{defn}
	Let $A$ be a $C_{0}(X)$-algebra, and let $p:Y\rightarrow X$ be a continuous map of locally compact Hausdorff spaces.  Then the \textbf{pullback} of $A$ by $p$ is the $C_{0}(Y)$-algebra $p^{*}A:=C_{0}(Y)\otimes_{p,C_{0}(X)}A$, where we take the balanced tensor product by regarding the $C_{0}(X)$-algebras $C_{0}(Y)$ and $A$ as $C_{0}(X)$-modules.  If there is no ambiguity about the map $p$, it will often be omitted from the notation, so that $p^{*}A = C_{0}(Y)\otimes_{C_{0}(X)}A$.
\end{defn}

It is easy to check that if $A$ is a $C_{0}(X)$-algebra and $p:Y\rightarrow X$ is a continuous map of locally compact Hausdorff spaces, the fibre over $y\in Y$ of $p^{*}A$ is $A_{p(y)}$.  Equipped with the notion of pullbacks, we can define what is meant by a $\GG$-algebra.

\begin{defn}
	Let $A$ be a $C_{0}(X)$-algebra.  A \textbf{$\GG$-action} on $A$ is a family $\alpha = \{\alpha^{i}:s_{i}^{*}A\rightarrow r_{i}^{*}A\}_{i\in I}$ of grading-preserving $C_{0}(U_{i})$-isomorphisms, such that $\alpha^{i}|_{s|_{U_{i}\cap U_{j}}^{*}A} = \alpha^{j}|_{s|_{U_{i}\cap U_{j}}^{*}A}$ for all $i,j\in I$, and such that the induced homomorphisms $\alpha_{u}:A_{s(u)}\rightarrow A_{r(u)}$ satisfy $\alpha_{uv} = \alpha_{u}\circ\alpha_{v}$.  If $A$ admits a $\GG$-action $\alpha$, we call $(A,\alpha)$ a \textbf{$\GG$-algebra}.
\end{defn}

The simplest nontrivial example of a $\GG$-algebra is $C_{0}(Y)$, where $Y$ is a $\GG$-space with anchor map $p:Y\rightarrow X$, and where $C_{0}(Y)$ is equipped with the $\GG$-action
\[
	\alpha_{u}(f)(y):=f(u^{-1}\cdot y)
\]
for all $u\in \GG$ and $f\in C_{0}(Y_{r(u)})$.

Now suppose that $E$ is a Hilbert module over a $\GG$-algebra $A$.  For $x\in X$, we can consider the fibre $E_{x}:=E\otimes_{A}A_{x}$, which is a Hilbert $A_{x}$-module, and if $p:Y\rightarrow X$ is a continuous map of locally compact Hausdorff spaces, we can consider the pullback $p^{*}E:=E\otimes_{A}p^{*}A$, which is a Hilbert $p^{*}A$-module.  If $T$ is an $A$-linear operator on $E$, we let $p^{*}T:=T\otimes 1_{p^{*}A}$ be its pullback to a $p^{*}A$-linear operator on $p^{*}E$.

\begin{defn}
	Let $(A,\alpha)$ be a $\GG$-algebra, and let $E$ be a 
	$\ZB_2$-graded Hilbert $A$-module.  A \textbf{$\GG$-action} on $E$ consists of a family $W = \{W^{i}:s_{i}^{*}E\rightarrow r_{i}^{*}E\}_{i\in I}$ of grading-preserving isometric Banach space isomorphisms, such that $W^{i}|_{s|_{U_{i}\cap U_{j}}^{*}E} = W^{j}|_{s|_{U_{i}\cap U_{j}}^{*}E}$ for all $i,j\in I$, and such that the induced isomorphisms $W_{u}:E_{s(u)}\rightarrow E_{r(u)}$ on the fibres satisfy $W_{uv} = W_{u}\circ W_{v}$, $\langle W_{u}\rho_{1},W_{u}\rho_{2}\rangle_{r(u)} = \alpha_{u}(\langle\rho_{1},\rho_{2}\rangle_{s(u)})$ and $W_{u}(\rho\cdot a) = W_{u}(\rho)\cdot\alpha_{u}(a)$ for all $(u,v)\in \GG^{(2)}$, $a\in A_{s(u)}$ and $\rho,\rho_{1},\rho_{2}\in E_{s(u)}$.  If $E$ admits a $\GG$-action $W$, we call $(E,W)$ a \textbf{$\GG$-Hilbert $A$-module}.
\end{defn}

If $V\rightarrow Y$ is a $\GG$-equivariant Hermitian vector bundle over a $\GG$-space $Y$, then the continuous sections vanishing at infinity $\Gamma_{0}(Y;V)$ of $V$ over $Y$ is a $\GG$-Hilbert $C_{0}(Y)$-module, with pointwise inner product and right action by $C_{0}(Y)$, and with $\GG$-action defined by
\begin{equation}\label{genhilmod}
	(W_{u}\rho)(y):=u_{*} \rho(u^{-1}\cdot y)
\end{equation}
for all $\rho\in\Gamma_{0}(Y_{r(u)};V|_{Y_{r(u)}})$.  All $\GG$-Hilbert module constructions in this paper will arise from some variant of the action \eqref{genhilmod}.

\begin{defn}
	If $B$ is a $\GG$-algebra, and $\pi:A\rightarrow\LL(E)$ is a representation of a $\GG$-algebra $(A,\alpha)$ on a $\GG$-Hilbert $B$-module $(E,W)$, we say that $\pi$ is \textbf{equivariant} if for all $i\in I$ we have
	\[
	\Ad_{W^{i}}(\pi^{s}_{i}(a)) = \pi^{r}_{i}(\alpha^{i}(a))
	\]
	for all $a\in A$.  Here $\pi^{s}_{i}:= 1_{C_{b}(U_{i})}\otimes\pi$ and $\pi^{r}_{i}:= 1_{C_{b}(U_{i})}\otimes\pi$ are respectively the induced homomorphisms $s_{i}^{*}A = C_{0}(U_{i})\otimes_{s,C_{0}(X)}A\rightarrow\LL(s^{*}_{i}E)$ and $r_{i}^{*}A = C_{0}(U_{i})\otimes_{r,C_{0}(X)}A\rightarrow\LL(r^{*}_{i}E)$.
\end{defn}

The definition of the equivariant $KK$-groups now follows in the usual way.

\begin{defn}	\label{gekkdefn}
Let $(A,\alpha)$ and $(B,\beta)$ be $\GG$-$C^{*}$-algebras.  
A \textbf{$\GG$-equivariant Kasparov $A$-$B$-module} is a 
triple $(A,{}_\pi E_B,F)$, where $(E,W)$ is a $\GG$-equivariant 
Hilbert $B$-module carrying an equivariant representation 
$\pi:A\rightarrow\LL(E)$, and where $F\in\LL(E)$ is 
homogeneous of degree 1 such that for all $a\in A$ one has
		\begin{enumerate}
			\item $\pi(a)(F-F^{*})\in\KK(E)$,
			\item $\pi(a)(F^{2}-1)\in\KK(E)$,
			\item $[F,\pi(a)]\in\KK(E)$,
		\end{enumerate}
		and such that for all $i\in I$
		\begin{enumerate}[resume]
			\item $\pi^{r}_{i}(r_{i}^{*}(a))(r_{i}^{*}F - W^{i}\circ s_{i}^{*}F\circ (W^{i})^{-1})\in r_{i}^{*}\KK(E)$.
		\end{enumerate}
		We say that two $\GG$-equivariant Kasparov 
		$A$-$B$-modules $(A,{}_\pi E_B,F)$ and 
		$(A,{}_{\pi'}E_B',F')$ are \textbf{unitarily equivalent} if 
		there exists a $\GG$-equivariant unitary 
		$V:E\rightarrow E'$ of degree 0 such 
		that $VFV^{*} = F'$ and $V\pi(a)V^{*} = \pi'(a)$ 
		for all $a\in A$.  We denote by $\EB^{\GG}(A,B)$ the 
		set of all unitary equivalence classes of $\GG$-equivariant 
		Kasparov $A$-$B$-modules.
		
		A \textbf{homotopy} in $\EB^{\GG}(A,B)$ is an 
		element of $\EB^{\GG}(A,B[0,1])$, and we define 
		$KK^{\GG}(A,B)$ to be the set of homotopy equivalence 
		classes in $\EB^{\GG}(A,B)$.
		
		The direct sum of $\GG$-equivariant Kasparov 
		$A$-$B$-modules makes $KK^{\GG}(A,B)$ into an abelian group.
\end{defn}

We also need \emph{unbounded representatives} of 
equivariant $KK$-classes.  The definition for such 
representatives is the natural extension of that due 
to Pierrot \cite{pierrot} to the locally Hausdorff case.  
We remark here that if $\AA$ is a dense $*$-subalgebra 
of a $C_{0}(X)$-algebra $A$, then we will assume that 
$C_{0}(X)\cdot \AA \subset\AA$, which will  be true in 
our examples.  We will denote by $\AA_{x}:=\AA/I_{x}\AA$ 
the fibre over $x\in X$, where as before $I_{x}$ is the 
kernel of the evaluation functional $f\mapsto f(x)$ on $C_{0}(X)$.

\begin{defn}
\label{unbdd}
	Let $A$ and $B$ be $\GG$-algebras.  An 
	\textbf{unbounded $\GG$-equivariant Kasparov $A$-$B$-module} 
	is a triple $(\AA,{}_{\pi}E,D)$, where $(E,W)$ is a 
	$\GG$-Hilbert $B$-module carrying an equivariant 
	representation $\pi$ of $A$ in $\LL(E)$, $D$ is a densely 
	defined, odd, unbounded, self-adjoint and regular operator 
	on $E$ commuting with the right action of $B$, and 
	where $\AA$ is a dense $*$-subalgebra of $A$ 
	preserved by the action of $\GG$ such that for all $a\in \AA$ one has:
	\begin{enumerate}
		\item $\pi(a)\Dom(D)\subset\Dom(D)$,
		\item $[D,\pi(a)]$ extends to an element of $\LL(E)$,
		\item $\pi(a)(1+D^{2})^{-\frac{1}{2}}\in\KK(E)$,
	\end{enumerate}
	and such that for all $i\in I$, $a\in \AA$ and $f\in C_{c}(U_{i})$ 
	one has
	\begin{enumerate}[resume]
		\item $f\cdot \pi^{r}_{i}(r_{i}^{*}(a))\cdot (r_{i}^{*}D - W^{i}\circ s_{i}^{*}D\circ (W^{i})^{-1})$ extends to an element of $\LL(r_{i}^{*}E)$ and
		\item $\Dom((r_{i}^{*}D)f) = W^{i}\Dom((s_{i}^{*}D)f)$.
	\end{enumerate}
\end{defn}

That all unbounded equivariant Kasparov modules define 
classes in $KK^{\GG}$ is an easy consequence of the 
corresponding result by Pierrot for Hausdorff groupoids.

\begin{prop}
	Let $A$ and $B$ be $\GG$-algebras, and let 
	$(\AA,{}_{\pi}E,D)$ be an unbounded $\GG$-equivariant 
	Kasparov $A$-$B$-module.  Then 
	$(A,{}_\pi E,D(1+D^2)^{-\frac{1}{2}})$ is 
	a $\GG$-equivariant Kasparov $A$-$B$-module.
\end{prop}

\begin{proof}
	That the first three requirements of Definition \ref{gekkdefn} 
	are met by $(A,{}_\pi E,D(1+D^{2})^{-\frac{1}{2}})$ is a 
	consequence of the corresponding result in the nonequivariant case \cite{baajjulg1}.  That the fourth requirement is met is a consequence of 
	restricting the corresponding result of Pierrot \cite[Th\'{e}or\`{e}me 6]{pierrot} 
	to each of the Hausdorff open subsets $U_{i}$ of $\GG$.
\end{proof}


We now come to the descent map in equivariant $KK$-theory, for which we need to discuss groupoid crossed products.  We will assume for this that $\GG$ comes equipped with a bundle $\Omega^{\frac{1}{2}}\rightarrow \GG$ of leafwise half-densities, as in \cite[Chapter 2.8]{ncg}.  Regard a $C_{0}(X)$-algebra $A$ as the continuous sections vanishing vanishing at infinity $\Gamma_{0}(X;\mathfrak{A})$ of the upper-semicontinuous bundle $\mathfrak{A}\rightarrow X$ of $C^{*}$-algebras whose fibre over $x\in X$ is $A_{x}$ \cite{legall,muhwil}.  Thus a $\GG$-algebra $(A,\alpha)$ can be regarded as the continuous sections vanishing at infinity of the $\GG$-space $\mathfrak{A}$ over $X$, where $\alpha_{u}:A_{s(u)}\rightarrow A_{r(u)}$ determines the action of $\GG$ on the bundle $\mathfrak{A}$.  

Define $\Gamma_{c}(\GG;r^{*}\mathfrak{A}\otimes\Omega^{\frac{1}{2}})$ to be the space of finite linear combinations of sections of the bundle $r^{*}\mathfrak{A}\otimes\Omega^{\frac{1}{2}}\rightarrow \GG$ which have compact support and are continuous in one of the $U_{i}$.  The space $\Gamma_{c}(\GG;r^{*}\mathfrak{A}\otimes\Omega^{\frac{1}{2}})$ is a $*$-algebra equipped with the convolution product
\[
	(f*g)_{u}:=\int_{v\in \GG^{r(u)}} f_{v}\alpha_{v}(g_{v^{-1}u})
\quad
\mbox{and with involution}
\quad
	(f^{*})_{u}:=\alpha_{u}((f_{u^{-1}})^{*}).
\]
The appropriate completion of $\Gamma_{c}(\GG;r^{*}\mathfrak{A}\otimes\Omega^{\frac{1}{2}})$ to a reduced $C^{*}$-algebra $A\rtimes_{r}\GG$ has been given in \cite[Section 3.7]{koshskand}.

In a similar manner, if $A$ is a $\GG$-algebra we can regard any $\GG$-Hilbert $A$-module $E$ as the continuous sections vanishing at infinity of an upper-semicontinuous bundle $\mathfrak{E}\rightarrow X$ whose fibre over $x\in X$ is $E_{x}$.  We define $\Gamma_{c}(\GG;r^{*}\mathfrak{E}\otimes\Omega^{\frac{1}{2}})$ to be the space of finite linear combinations of sections of the bundle $r^{*}\mathfrak{E}\otimes\Omega^{\frac{1}{2}}\rightarrow \GG$ that have compact support and are continuous in one of the $U_{i}$.  The formulae
\[
	\langle\rho^{1},\rho^{2}\rangle^{\GG}_{u}:=\int_{v\in \GG^{r(u)}} \alpha_{v}\langle \rho^{1}_{v^{-1}},\rho^{2}_{v^{-1}u}\rangle
\quad
\mbox{and}
\quad
	(\rho\cdot f)_{u}:=\int_{v\in \GG^{r(u)}}\rho_{v}\alpha_{v}(f_{v^{-1}u})
\]
defined for $\rho^{1},\rho^{2},\rho\in \Gamma_{c}(\GG;r^{*}\mathfrak{E}\otimes\Omega^{\frac{1}{2}})$ and $f\in\Gamma_{c}(\GG;r^{*}\mathfrak{A}\otimes\Omega^{\frac{1}{2}})$ determine an $A\rtimes_{r}\GG$-valued inner product and right action respectively on $\Gamma_{c}(\GG;r^{*}\mathfrak{E}\otimes\Omega^{\frac{1}{2}})$, and we may complete in the norm arising from $\langle\cdot,\cdot\rangle^{\GG}$ to obtain a Hilbert $A\rtimes_{r}\GG$-module which we denote by $E\rtimes_{r}\GG$.  If $T$ is an $A$-linear operator on $E$, we denote by $\mathfrak{dom}(T)$ the bundle over $X$ whose fibre over $x\in X$ is $\Dom(T)\otimes_{A}A_{x}$.  Then as in \cite[D\'{e}finition 2, Proposition 3]{pierrot}
we define $r^{*}(T)$ on $\Gamma_{c}(\GG;r^{*}\mathfrak{dom}(T)\otimes\Omega^{\frac{1}{2}})$ by
\[
	(r^{*}(T)\rho)_{u}:=T_{r(u)}\rho_{u}.
\]
If $T\in\LL(E)$ one can use the norm of $T$ to bound that of $r^{*}(T)$, and then one can use $T^{*}$ to show that $r^{*}(T)\in\LL(E\rtimes_{r}\GG)$.

\begin{lemma}\label{anal}
	For any densely defined $A$-linear operator $T:\Dom(T)\rightarrow E$, we have $r^{*}(T^{*})\subset  r^{*}(T)^{*}$.  Moreover  
	$\overline{r^{*}(T^{*})} = r^{*}(T)^{*}$.
\end{lemma}

\begin{proof}
	Fix $\xi\in\Dom(r^{*}(T^{*})) =\Gamma_{c}(\GG;r^{*}\mathfrak{dom}(T^{*})\otimes\Omega^{\frac{1}{2}})$, and assume without loss of generality that $\xi$ has compact support in some Hausdorff open subset $U_{i}$ of $\GG$.  For each $u\in \GG$, use the fact that $\xi_{u}\in\mathfrak{dom}(T^{*})_{r(u)}\otimes\Omega^{\frac{1}{2}}_{u}$ to define a section $\eta$ of $r^{*}\mathfrak{E}\otimes\Omega^{\frac{1}{2}}\rightarrow \GG$ by
	\[
	\eta_{u}:=T^{*}_{r(u)}\xi_{u}.
	\]
	Since $\xi$ is continuous with compact support in $U_{i}$ so too is $\eta$, thus $\eta\in\Gamma_{c}(\GG,r^{*}\mathfrak{E}\otimes\Omega^{\frac{1}{2}})$.  For any $\rho\in\Dom(r^{*}(T)) = \Gamma_{c}(\GG;r^{*}\mathfrak{dom}(T)\otimes\Omega^{\frac{1}{2}})$ we can then calculate
	\[
	\langle\xi,r^{*}(T)\rho\rangle^{\GG}_{u} = \int_{v\in \GG^{r(u)}}\alpha_{v}(\langle\xi_{v^{-1}},T_{s(v)}\rho_{v^{-1}u}\rangle) = \int_{v\in \GG^{r(u)}}\alpha_{v}(\langle T^{*}_{s(v)}\xi_{v^{-1}},\rho_{v^{-1}u}\rangle) = \langle\eta,\rho\rangle^{\GG}_{u}
	\]
	for all $u\in \GG$, so that $\xi\in\Dom(r^{*}(T)^{*})$.  The above calculation also shows that $r^{*}(T)^{*}\xi = \eta = r^{*}(T^{*})\xi$, so that we indeed have $r^{*}(T^{*})\subset r^{*}(T)^{*}$.  
	
	Fix $\xi\in\Dom(r^{*}(T)^{*})$.
	We show that $\xi\in\overline{r^{*}(T^{*})}$.  Let $\{\xi^{n}\}_{n\in\NB}\subset\Gamma_{c}(\GG;r^{*}\mathfrak{dom}(T^*)\otimes\Omega^{\frac{1}{2}})$ 
	be a sequence converging in $E\rtimes_{r}\GG$ to $\xi$.  
	Then the sequence $\{\langle\xi^{n},r^{*}(T)\rho\rangle^{\GG}\}_{n\in\NB}$ 
	of elements of 
	$\Gamma_{c}(\GG;r^{*}\mathfrak{A}\otimes\Omega^{\frac{1}{2}})$ 
	defined for $u\in \GG$ by
	\begin{equation}
	\label{bumps}
	\langle\xi^{n},r^{*}(T)\rho\rangle_{u}^{\GG} = \int_{v\in \GG^{r(u)}}\alpha_{v}(\langle\xi^{n}_{v^{-1}},T_{s(v)}\rho_{v^{-1}u}\rangle) = \int_{v\in \GG^{r(u)}}\alpha_{v}(\langle T_{s(v)}^*\xi^{n}_{v^{-1}},\rho_{v^{-1}u}\rangle)
	\end{equation}
	converges in $A\rtimes_{r}\GG$ for all 
	$\rho\in\Gamma_{c}(\GG;r^{*}\mathfrak{dom}(T)\otimes\Omega^{\frac{1}{2}})$.
	 For each $v\in \GG^{r(u)}$ one can on the right hand side of \eqref{bumps} 
	 take bump functions $\rho$ with support of decreasing radius 
	 about $v^{-1}u$ to show that we have convergence of 
	 $\{(r^{*}(T^{*})\xi^{n})_{v^{-1}} = T_{s(v)}^*\xi^{n}_{v^{-1}}\}_{n\in\NB}$ 
	 to an element of $E_{s(v)}$, and doing this for all $v\in \GG^{r(u)}$ 
	 and all $u\in \GG$ shows that in fact $\{r^{*}(T^{*})\xi^{n}\}_{n\in\NB}$ 
	 converges in $E\rtimes_{r}\GG$, implying that $\xi^{n}\rightarrow \xi$ 
	 in the graph norm on $\Dom(r^{*}(T^{*}))$ as claimed.
\end{proof}

Finally, we observe that if $A$ and $B$ are $\GG$-algebras, and if $(E,W)$ is a $\GG$-Hilbert $B$-module with an equivariant representation $\pi:A\rightarrow\LL(E)$, then the formula
\[
	((\pi\rtimes_{r} \GG)(f)\rho)_{u}
	:=\int_{v\in \GG^{r(u)}} \pi(f_{v})W_{v}(\rho_{v^{-1}u})
\]
defined for $f\in\Gamma_{c}(\GG;r^{*}\mathfrak{A}\otimes\Omega^{\frac{1}{2}})$ and $\rho\in\Gamma_{c}(\GG;r^{*}\mathfrak{E}\otimes\Omega^{\frac{1}{2}})$ determines a representation $\pi\rtimes_{r}\GG:A\rtimes_{r}\GG\rightarrow \LL(E\rtimes_{r}\GG)$.

\begin{prop}\label{descent}
	Let $A$ and $B$ be $\GG$-algebras, 
	and let $(\AA,{}_{\pi}E,D)$ be a $\GG$-equivariant unbounded 
	Kasparov $A$-$B$-module.  Let $\widetilde{\AA}$ denote the 
	bundle of $*$-algebras over $X$ whose fibre over 
	$x\in X$ is $\AA_{x}$.  Then
	\[
		(\Gamma_{c}(\GG;r^{*}\widetilde{\AA}\otimes\Omega^{\frac{1}{2}}),
		{}_{\pi\rtimes_{r}\GG}E\rtimes_{r}\GG,r^{*}(D))
	\] 
	is an unbounded Kasparov $A\rtimes_{r}\GG$-$B\rtimes_{r}\GG$-module.
\end{prop}

\begin{proof}
	Since $D$ is odd for the grading of $E$, $r^{*}(D)$ is 
	odd for the induced grading of $E\rtimes_{r}\GG$.  Symmetry of $D$ gives symmetry of $r^{*}(D)$, so without loss of generality we may assume that $r^{*}(D)$ is closed.  Self-adjointness of $r^{*}(D)$ is then a consequence of the self-adjointness of $D$ together with Lemma \ref{anal}.
	
	Regularity of $r^{*}(D)$ is a  consequence of that of $D$.  Indeed, for any $\rho\in\Gamma_{c}(\GG;r^{*}\mathfrak{dom}(D)\otimes\Omega^{\frac{1}{2}})$ we have
	\[
	((1+r^{*}(D)^{2})\rho)_{u} = (1_{r(u)}+D_{r(u)}^{2})\rho_{u}.
	\]
	Hence the range of the operator
	$(1+r^{*}(D)^{2})$ when restricted to $\Gamma_{c}(\GG;r^{*}\mathfrak{dom}(D)\otimes\Omega^{\frac{1}{2}})$ is $\Gamma_{c}(\GG;r^{*}\mathfrak{range}(1+D^{2})\otimes\Omega^{\frac{1}{2}})$, where $\mathfrak{range}(1+D^{2})$ denotes the bundle over $X$ whose fibre over $x\in X$ is $\Range(1+D^{2})\otimes_{A} A_{x}$, which by regularity of $D$ is dense in $E_{x} = E\otimes_{A}A_{x}$.  Thus the range of $(1+r^{*}(D)^{2})$ contains the dense subspace $\Gamma_{c}(\GG;r^{*}\mathfrak{range}(1+D^{2})\otimes\Omega^{\frac{1}{2}})$ of $E\rtimes_{r}\GG$, and it follows that $r^{*}(D)$ is regular.
	
	Regarding commutators, a simple calculation tells us that for any $u\in\GG$, the vector $([r^{*}(D),(\pi\rtimes_{r}\GG)(f)]\rho)_{u}$ is equal to
	\begin{align*}
	 \int_{v\in \GG^{r(u)}}\bigg([D_{r(u)},\pi(f_{v})] +\pi(f_{v})\big(D_{r(v)}-W_{v}\circ D_{s(v)}\circ W_{v^{-1}}\big)\bigg)(W_{v}\rho_{v^{-1}u})
	\end{align*}
	for all 
	$\rho\in\Gamma_{c}(\GG;r^{*}\mathfrak{dom}(T)\otimes\Omega^{\frac{1}{2}})$,
	so Properties 2 and 4 in Definition \ref{unbdd} imply
	that the operator $[r^{*}(D),(\pi\rtimes_{r}\GG)(f)]$ extends to an element of $\LL(E\rtimes_{r}\GG)$, with adjoint $[r^{*}(D),(\pi\rtimes_{r}\GG)(f^{*})]$.
	
	The only thing that remains to check is compactness 
	of $(\pi\rtimes_{r}\GG)(f)(1+r^{*}(D)^{2})^{-\frac{1}{2}}$ for $f\in\Gamma_{c}(\GG;r^{*}\tilde{\AA}\otimes\Omega^{\frac{1}{2}})$.  For any $\rho\in\Gamma_{c}(\GG;r^{*}\mathfrak{E}\otimes\Omega^{\frac{1}{2}})$ the definition of $r^*(D)$ gives
	\begin{align*}
	((1+r^{*}(D)^{2})^{-\frac{1}{2}}(\pi\rtimes_{r}\GG)(f^{*})\rho)_{u} =& (1+D_{r(u)}^{2})^{-\frac{1}{2}}\int_{v\in \GG^{r(u)}}\pi((f)^{*}_{v})W_{v}(\rho_{v^{-1}u})\\ =&\int_{v\in \GG^{r(u)}}(1+D_{r(v)}^{2})^{-\frac{1}{2}}\pi((f)^{*}_{v})W_{v}(\rho_{v^{-1}u}),
	\end{align*}
	and since $(1+D_{r(v)}^{2})^{-\frac{1}{2}}\pi((f)^{*}_{v})\in\KK(E)_{r(v)}$ for all $v\in \GG^{r(u)}$ by Property 3 in Definition \ref{unbdd}, it follows that $(1+r^{*}(D)^{2})^{-\frac{1}{2}}(\pi\rtimes_{r}\GG)(f^{*})$ is an element of $\Gamma_{c}(\GG;r^{*}\KK(E)\otimes\Omega^{\frac{1}{2}})$. 
	A similar argument to the one used in \cite[Page 172]{KasparovEqvar} 
	then tells us that $(1+r^{*}(D))^{-\frac{1}{2}}(\pi\rtimes_{r}\GG)(f^{*})$ 
	can be approximated by finite rank operators
	on $E\rtimes_r\GG$ so
	is an element of $\KK(E\rtimes_{r}\GG)$, and hence so 
	too is its adjoint $(\pi\rtimes_{r}\GG)(f)(1+r^{*}(D)^{2})^{-\frac{1}{2}}$.
\end{proof}

Let us remark finally that if $Y$ is a locally compact Hausdorff $\GG$-space, with corresponding bundle $C_{0}(\mathfrak{Y})\rightarrow X$ whose fibre over $x\in X$ is $C_{0}(Y_{x})$, then we have an inclusion $\Gamma_{c}(Y\rtimes \GG;\Omega^{\frac{1}{2}})\ni f\mapsto\tilde{f}\in\Gamma_{c}(\GG;r^{*}C_{0}(\mathfrak{Y})\otimes\Omega^{\frac{1}{2}})$ defined by
\[
\tilde{f}_{u}(y):=f(y,u).
\]
For ease of notation we will usually just refer to $\tilde{f}$ as $f$.  By density of $C_{c}(Y_{x})$ in $C_{0}(Y_{x})$ for each $x\in X$, the subalgebra $\Gamma_{c}(Y\rtimes \GG;\Omega^{\frac{1}{2}})$ is dense in $C_{0}(Y)\rtimes_{r}\GG$.  We will use this fact in the construction of our Godbillon-Vey spectral triple.

\subsection{Semifinite spectral triples}

One of the defining features of a spectral triple $(\AA,\HH,\DD)$ 
is that the operators $a(1+\DD^2)^{-\frac{1}{2}}$ are contained 
in the compact operators $\KK(\HH)$ for all $a\in\AA$.  These 
compact operators come equipped with a 
trace $\Tr$, which is used to measure the rank of projections
that appear in the definition of the index, and subsequent
index formulae \cite{CM,higsonlocind}.


Semifinite spectral triples are a generalisation of spectral triples for which the rank of projections
is measured by a different trace.
More precisely we require a  faithful normal semifinite trace $\tau$
on a semifinite von Neumann algebra $\NN\subset \BB(\HH)$.
We denote by $\KK_{\tau}(\NN)$ the norm-closed 
ideal in $\NN$ generated by projections of finite $\tau$ -trace, 
and refer to $\KK_{\tau}(\NN)$ as the ideal 
of \emph{$\tau$-compact operators}, \cite{fackkosaki}.
 
%
%
%
%

\begin{defn}
Let $(\NN,\tau)$ be a semifinite von Neumann algebra, 
regarded as an algebra of operators on a Hilbert space $\HH$.  
A \textbf{semifinite spectral triple relative to $(\NN,\tau)$} 
is a triple $(\AA,{}_\pi\HH,\DD)$ consisting of a 
$*$-algebra $\AA$ represented in $\NN$ by $\pi:\AA\to\NN\subset\BB(\HH)$, 
and a densely-defined, unbounded, self-adjoint 
operator $\DD$ affiliated to $\NN$ such that
\begin{enumerate}
    \item $\pi(a)\Dom(\DD)\subset\Dom(\DD)$ so that $[\DD,\pi(a)]$ 
    is densely defined, and moreover extends to a bounded 
    operator on $\HH$ for all $a\in \AA$,
    \item $\pi(a)(1+\DD^2)^{-\frac{1}{2}}\in\KK_{\tau}(\NN)$ 
    for all $a\in\AA$.
\end{enumerate}
We say that $(\AA,{}_\pi\HH,\DD)$ is \textbf{even} if $\AA$ is 
even and $\DD$ is odd for some $\ZB_2$-grading on $\HH$, and otherwise 
we call $(\AA,{}_\pi\HH,\DD)$ \textbf{odd}.
\end{defn}

Connes' original notion of spectral triple defines a 
subclass of semifinite spectral triples, for which 
$(\NN,\tau) = (\BB(\HH),\Tr)$.  Just as  the bounded transform
of a 
spectral triple $(\AA,{}_\pi\HH,\DD)$ defines 
a Fredholm module (over 
the $C^{*}$-completion $A$ of $\AA$), and hence a 
class in $KK_{*}(A,\CB)$, semifinite spectral triples 
have a close relationship with $KK$-theory.  

To 
see this, we first suppose that  $B$ is a $C^{*}$-algebra, 
$X_{B}$ is a Hilbert $B$-module with 
inner product $\langle\cdot,\cdot\rangle_B$, and $\tau$ is a faithful 
norm lower semicontinuous semifinite trace on $B$. We can form
the GNS space $L^2(B,\tau)$, or $L^2(X,\tau)$ with inner product
$( x|y)=\tau(\langle x,y\rangle_B)$. These two Hilbert spaces
are related by $X\otimes_B L^2(B,\tau)\cong L^2(X,\tau)$.

Then by results in \cite{LN}, we obtain a faithful 
normal semifinite trace $\Tr_{\tau}$, called the \emph{dual trace}, 
on the weak closure $\NN=\End_B(X)''\subset \BB(L^2(X_B,\tau))$ 
of the adjointable $B$-linear 
operators on $X_{B}$. The functional $\Tr_\tau$ satisfies
\[
    \Tr_{\tau}(\Theta_{\xi,\eta}):=\tau(\langle\eta,\xi\rangle_{B}).
\]

\begin{prop}
\label{sem1}
\label{sfst}
Let $(\AA,{}_\pi X_{B},\DD)$ be an even (resp. odd) 
unbounded Kasparov $A$-$B$ module, and 
suppose that $\tau$ is a faithful norm lower semicontinuous 
semifinite 
trace on $B$.  Let $(\NN,\Tr_{\tau})$ be the semifinite von 
Neumann algebra obtained from $X_{B}$ and $\tau$ as above.  
Then (with a slight abuse of notation)
$$
(\AA,{}_{\pi\hotimes1} X_{B}\hotimes_{B}L^{2}(B,\tau),\DD\hotimes1)
=(\AA,{}_\pi L^{2}(X_B,\tau),\DD)
$$ 
is an even (resp. odd) semifinite spectral triple 
relative to $(\NN,\Tr_{\tau})$.
\end{prop}

\begin{proof}
Clearly $\AA\subset\NN$, and the commutant of
$\NN$ is just $B''$. Since $\DD$ is 
$B$-linear, every unitary in $B''$
preserves the domain of
$\DD\hotimes1$, whence $\DD\hotimes 1$ 
is affiliated to $\NN$.  That $[\DD\hotimes 1,\pi(a)\hotimes 1]$ 
is bounded for all $a\in\AA$ is a consequence 
of the corresponding fact for the Kasparov module $(\AA,{}_\pi X_{B},\DD)$, 
and that $(\pi(a)\hotimes1)(1+\DD\hotimes1^2)^{-\frac{1}{2}}$ is $\tau$-compact 
is true because the algebra $\KK(X_{B})$ is 
contained in $\KK_{\tau}(\NN)$ by construction.
\end{proof}

In fact, a converse to Proposition \ref{sfst} is also true: 
namely, every semifinite spectral triple can be factorised 
into a $KK$-class and a trace \cite{KNR}.  
Although we will not need this converse result, 
it provides a useful way of thinking about semifinite spectral triples.

One of the most useful features of (nice) spectral triples is 
that their pairing with $K$-theory can be computed using 
the local index formula, \cite{CM}.  The same is true for (nice)
semifinite spectral triples.  There are now numerous results  
generalising the Connes-Moscovici local index formula for spectral triples to 
semifinite spectral triples \cite{benfack,cprs1,cprs2,cprs3,cprs4,cgrs1,cgrs2}.

%

\section{Construction of the Kasparov modules}

In this section, $(M,\FF)$ will denote a transversely 
orientable foliated manifold of codimension $q$, 
with holonomy groupoid $\GG$ and normal bundle 
$N = TM/T\FF\rightarrow M$.  The normal bundle is a 
$\GG$-equivariant vector bundle, as explained at the end of Section 2.2, and for $u\in \GG$ we 
let $u_{*}:N_{s(u)}\rightarrow N_{r(u)}$ be the 
corresponding map $n\mapsto u_{*}n$.  We assume $\GG$ to be equipped with a countable cover $\mathcal{U}:=\{U_{i}\}_{i\in I}$ by Hausdorff open subsets. 
We do not assume $K$-orientability at any point, 
working with exterior algebra bundles instead of spinor bundles.  


The first of the two constructions,  the Connes fibration, 
will not feature in the index theorem in the final section.  
The  Kasparov module of the Connes fibration  provides
a Thom-type isomorphism which does not conceptually 
affect our final index formulae.
We include the Connes fibration for the sake of 
completeness, and to show that the whole construction does 
indeed factor through groupoid equivariant $KK$-theory.

\subsection{The Connes fibration}

We begin this section with a revision of a construction 
due to Connes \cite{cyctrans}.  Connes starts with an 
oriented manifold $M$ of dimension $n$ with an action of 
a discrete group $\Gamma$ of orientation-preserving diffeomorphisms.  Such a setting
provides an \'{e}tale model of the transverse geometry of a transversely oriented foliation.

Connes shows that if $W\rightarrow M$ denotes the ``bundle 
of Euclidean metrics" for the tangent bundle $TM$ over 
$M$, then one can construct a dual Dirac 
class in $KK^{\Gamma}_{\frac{n(n+1)}{2}}(C_{0}(M),C_{0}(W))$.  
The manifold $W$ has the advantage that the pullback of $TM$ to $W$ admits a $\Gamma$-invariant Euclidean metric, even 
though one need not exist on $M$ in general.  We  
show that Connes' construction can be carried out
directly in the 
groupoid equivariant setting, as it may be useful for future 
work in constructing the Godbillon-Vey invariant as a 
semifinite spectral triple in arbitrary codimension.

We let $\pi_{F}:F^{+}(N)\rightarrow M$ be the 
principal $GL^{+}(q,\RB)$-bundle of positively oriented frames 
for the vector bundle $N\rightarrow M$, whose 
fibre $(F^{+}(N))_{x}$ over $x\in M$ consists  
of positively oriented linear isomorphisms 
$\phi:\mathbb{R}^{q}\rightarrow N_{x}$.  Then $F^{+}(N)$ 
is a $\GG$-space with anchor map $\pi_{F}:F^{+}(N)\rightarrow M$ 
and action defined by
\begin{equation}
\label{Delta}
    u\cdot \phi:= u_{*}\circ\phi:\RB^{q}\rightarrow N_{r(u)}
\end{equation}
for $\phi:\RB^{q}\rightarrow N_{s(u)}$ in $F^{+}(N)_{s(u)}$.  
Observe that this action of $\GG$ commutes with the right 
action of $GL^{+}(q,\RB)$ on the principal $GL^{+}(q,\RB)$-bundle $F^{+}(N)\rightarrow M$.

The vertical subbundle $\ker(d\pi_F)=VF^{+}(N)\rightarrow F^+N$ of $TF^{+}(N)$ admits a trivialisation $VF^{+}(N)\rightarrow F^{+}(N)\times\mathfrak{gl}(q,\RB)$,
where $\mathfrak{gl}(q,\RB) = M_{q}(\RB)$ is the Lie algebra 
of $GL^{+}(q,\RB)$ consisting of all $q\times q$ real matrices. The trivialisation is given by the formula
\[
    F^{+}(N)\times\mathfrak{gl}(q,\RB)\ni(\phi,v)\mapsto v_{\phi}
    :=\frac{d}{dt}(\phi\cdot\exp(tv))\bigg|_{t=0}\in VF^{+}(N).
\]
For $u\in \GG$, the differential $u_{*}:VF^{+}(N)_{s(u)}\rightarrow VF^{+}(N)_{r(u)}$ of $u\cdot:F^{+}(N)_{s(u)}\rightarrow F^{+}(N)_{r(u)}$ in the fibres defines on $VF^{+}(N)$ the structure of a $\GG$-equivariant vector bundle.  Since the left action of $\GG$ commutes with the right action of 
$GL^{+}(q,\RB)$, one has
\begin{equation}
    u_{*} v_{\phi} = \frac{d}{dt}(u\cdot(\phi\cdot\exp(tv))\bigg|_{t=0} 
    = \frac{d}{dt}((u\cdot\phi)\cdot\exp(tv))\bigg|_{t=0} = v_{u\cdot\phi}
\label{eq:d-dt}
\end{equation}
for all $\phi\in (F^{+}(N))_{s(u)}$, and so with respect to the 
trivialisation $F^{+}(N)\times\mathfrak{gl}(q,\RB)$ of $VF^{+}(N)$ 
we have
\begin{equation}
\label{uonv}
    u_{*}(\phi,v) = (u\cdot\phi,v).
\end{equation}
for all $\phi\in F^{+}(N)$ and $v\in\mathfrak{gl}(q,\RB)$.

Consider now the quotient $Q:=F^{+}(N)/SO(q,\RB)$ of 
$F^{+}(N)$ by the right action of $SO(q,\RB)$.  The projection 
$\pi_{F}:F^{+}(N)\rightarrow M$ descends to a projection 
$\pi_{Q}:Q\rightarrow M$, which defines a fibre bundle 
with typical fibre $S^{+}_{q}:=GL^{+}(q,\RB)/SO(q,\RB)$, 
the space of positive definite, symmetric $q\times q$ matrices. Moreover, since the action of $\GG$ on $F^{+}(N)$ commutes 
with the right action of $SO(q,\RB)$, it follows that $Q$ 
is a $\GG$-space with anchor map $\pi_{Q}:Q\rightarrow M$, 
and with action of $u\in \GG$ given by
\begin{equation}
\label{uonfibres}
u\cdot[\phi]:=[u\cdot\phi]=[u_{*}\circ\phi]
\end{equation}
for all $[\phi]\in Q_{s(u)}$.  Following \cite{ben2,zhang}, we refer to 
$\pi_{Q}:Q\rightarrow M$ as the Connes fibration.

\begin{defn}
	The fibre bundle $\pi_{Q}:Q\rightarrow M$ is a $\GG$-space 
	called the \textbf{Connes fibration} for the normal bundle $N$.
\end{defn}
 
Let us consider the geometry of the fibres of $Q\rightarrow M$.  Since $SO(q,\RB)$ is compact, the pair $(GL^{+}(q,\RB),SO(q,\RB))$ is a Riemannian symmetric pair and hence the space $S^{+}_{q}$ can be equipped with a $GL^{+}(q,\RB)$-invariant metric under which it is, by \cite[Proposition 3.4]{helgason}, a globally symmetric Riemannian space.  The Riemannian space $S_{q}^{+}$ is moreover of noncompact type, so by \cite[Theorem 3.1]{helgason} has everywhere non-positive sectional curvature.  We can find 
a locally finite open cover $\UU$ of $M$ by sets $U$ for 
which the vertical bundle $VQ|_{U}\cong U\times TS_{q}$, 
and then choosing 
a partition of unity subordinate to $\UU$ allows us to equip 
the bundle $VQ\rightarrow Q$ with a Euclidean structure.  
We will assume from here on that $VQ\rightarrow Q$ 
is equipped with a Euclidean structure in this way.

\begin{prop}
\label{verticalequi}
	The bundle $VQ\rightarrow Q$ is a $\GG$-equivariant 
	Euclidean bundle over the $\GG$-space $Q$. Consequently
	$\Cliff(VQ)$ and $\Cliff(V^*Q)$ are $\GG$-equivariant bundles.
\end{prop}

\begin{proof}
Fix $u\in \GG$ and suppose that $U_{s}$ and $U_{r}$ are 
open sets in $M$ 
containing $s(u)$ and $r(u)$ respectively, such that we 
have local trivialisations
$N|_{U_{s}} \cong U\times\RB^{q}$ and 
$N|_{U_{r}} \cong U\times\RB^{q}$, with respect to 
which the holonomy action $u_*:N_{s(u)}\rightarrow N_{r(u)}$ 
is the action on $\RB^{q}$ of an element $\tilde{u}\in GL^{+}(q,\RB)$.  

We obtain corresponding local trivialisations 
$F^{+}(N)|_{U_{s}}\cong U\times GL^{+}(q,\RB)$ 
and $F^{+}(N)|_{U_{r}}\cong U\times GL^{+}(q,\RB)$ 
of the local frame bundles over $U_{s}$ and $U_{r}$, 
in which the holonomy action $u\cdot:F^{+}(N)_{s(u)}\rightarrow F^{+}(N)_{r(u)}$ 
is left multiplication on $GL^{+}(q,\RB)$ by $\tilde{u}$, and taking the 
quotient by $SO(q,\RB)$ we get local trivialisations 
$Q|_{U_{s}}\cong U\times S_{q}^{+}$ and 
$Q|_{U_{r}}\cong U\times S_{q}^{+}$ in which 
$u\cdot:Q_{s(u)}\rightarrow Q_{r(u)}$ is the isometry of 
$S^{+}_{q}=GL^{+}(q,\RB)/SO(q,\RB)$ defined by left 
multiplication by $\tilde{u}\in GL^{+}(q,\RB)$. 
Thus $\GG$ acts by orientation-preserving isometries between the fibres of $Q$, inducing 
an action by special orthogonal transformations on the Euclidean 
bundle $VQ\rightarrow Q$ of vectors tangent to the fibres of 
$Q\rightarrow M$, hence making $VQ\rightarrow Q$ a 
$\GG$-equivariant Euclidean bundle over the $\GG$-space $Q$.
The final statement follows from functoriality of Clifford algebras with respect to orthogonal maps.
\end{proof}

That the fibres have nonpositive sectional curvature allows 
us to define a dual Dirac class for $Q$ over $M$ in a 
similar manner to Connes \cite{cyctrans}.  First, let $\CCl(V^{*}Q)$ be equipped with the $\GG$-structure 
arising from the action of $\GG$ on the equivariant bundle 
$\Cliff(V^{*}Q)$ over the $\GG$-space $Q$, denoted for $u\in \GG$ by $u_{\diamond}:\Cliff(V^{*}_{[\phi]}Q)\rightarrow \Cliff(V^{*}_{u\cdot[\phi]}Q)$ for all $[\phi]\in Q_{s(u)}$.  
That is, we define for any $u\in \GG$ an isomorphism 
$\alpha^{1}_{u}:\CCl(V^{*}Q|_{Q_{s(u)}})
\rightarrow \CCl(V^{*}Q|_{Q_{s(u)}})$ 
by
\begin{equation}
\label{actiononQ}
    \alpha^{1}_{u}(a)([\phi]):=u_{\diamond}a(u^{-1}\cdot[\phi])
\end{equation}
for all $[\phi]\in Q_{r(u)}$.  Also let 
\[
    E^{1}:=\Lambda^{*}(V^{*}Q)\otimes\CB
\]
be the complexified exterior algebra bundle of the  bundle 
of vertical covectors $V^{*}Q$ over $Q$.  Here we 
equip $V^{*}Q$ with the Euclidean structure coming 
from its dual $VQ$, which determines a Hermitian 
structure on $V^{*}Q\otimes\CB$ and hence on $E^{1}$.  
Observe that
\[
    X_{E^{1}}:=\Gamma_{0}(Q;E^{1})
\]
is a Hilbert $\CCl(V^{*}Q)$-module under the inner product
\[
	\langle\rho^{1},\rho^{2}\rangle_{\CCl(V^{*}Q)}([\phi])
	:=\psi_{V^{*}Q}(\rho^{1}([\phi]))\psi_{V^{*}Q}(\rho^{2}([\phi]))
\]
and right action
\[
	(\rho\cdot a)([\phi]):=c_{R}(a([\phi]))\rho([\phi]),
\]
where $c_{R}$ is the right action of $\Cliff(V^{*}Q)$ on the 
Clifford bimodule $E^{1}$.

The isometric action of $\GG$ on the Euclidean bundle 
$VQ$ over $Q$ gives rise to a unitary action of $\GG$ on
$E^{1}$, denoted for each $u\in \GG$ by $u_{*}:E^{1}_{[\phi]}\rightarrow E^{1}_{u\cdot[\phi]}$ for all $[\phi]\in Q_{s(u)}$, and hence determines an isomorphism 
$W^{1}_{u}:\Gamma_{0}(Q_{s(u)};E^{1}|_{Q_{s(u)}})
\rightarrow \Gamma_{0}(Q_{r(u)};E^{1}|_{Q_{r(u)}})$ 
of Banach spaces given by the formula
\[
    (W^1_{u}\rho)([\phi]):=u_{*}\rho(u^{-1}\cdot[\phi])
\]
for all $[\phi]\in Q_{r(u)}$.  A routine calculation using Lemma \ref{cliffacthom} shows that
\[
    \langle W^1_{u}\rho^{1},W^1_{u}\rho^{2}\rangle_{\CCl(V^{*}Q)} 
    = \alpha^{1}_{u}(\langle\rho^{1},\rho^{2}\rangle_{\CCl(V^{*}Q)}),
\]
so $(X_{E^{1}},W^1)$ is a $\GG$-equivariant 
Hilbert $\CCl(V^{*}Q)$-module.

Choose now a Euclidean metric for $N$.  Such a choice is determined by a section $\sigma:M\rightarrow Q$ of $\pi_{Q}:Q\rightarrow M$.  For $[\phi_{1}],[\phi_{2}]$ in the same fibre $Q_{x}$, denote by $h([\phi_{1}],[\phi_{2}])$ the geodesic distance between $[\phi_{1}]$ and $[\phi_{2}]$ in the fibre, and then for any $[\phi_{0}]\in Q$ let $h^{[\phi_{0}]}:Q\rightarrow\RB$ be the function
\[
    h^{[\phi_{0}]}([\phi]):=h([\phi_{0}],[\phi]).
\]
In particular, for $x\in M$ and $[\phi]\in Q_{x}$, $h^{\sigma(x)}([\phi])$ gives the distance in the fibre between $[\phi]$ and the section $\sigma$.  Consider now the vertical 1-form
\[
    Z_{[\phi]}:=h^{\sigma(\pi_{Q}([\phi]))}([\phi])dh^{\sigma(\pi_{Q}([\phi]))}_{[\phi]},
\]
where $d$ denotes the exterior derivative in the fibre.  Define an operator $B_{1}$ on the dense submodule $X^{c}_{E^{1}}:=\Gamma_{c}(Q;E^{1})$ of $X_{E^{1}}$ by the formula
\[
    (B_{1}\rho)([\phi]):=c_{L}(Z_{[\phi]})\rho([\phi]),
\]
where $c_{L}$ is the left representation of $\Cliff(V^{*}Q)$ on the Clifford bimodule $E^{1}$.  Since $c_{L}$ and $c_{R}$ commute, $B_{1}$ commutes with the right action of $\CCl(V^{*}Q)$.
Finally, we let $m$ be the representation of $C_{0}(M)$ on $X_{E^{1}}$ by multiplication, that is
\[
	(m(f)\rho)([\phi]):=f(\pi_{Q}([\phi]))\rho([\phi])
\]
for all $f\in C_{0}(M)$ and $\rho\in X_{E^{1}}$.  Equivariance of the map $\pi_{Q}$ tells us that $m$ is an equivariant representation.

\begin{prop}
\label{B1}
The triple $(C_{0}(M),{}_{m}X_{E^{1}},B_{1})$ 
is an unbounded $\GG$-equivariant Kasparov 
$C_{0}(M)$-$\CCl(V^{*}Q)$-module, hence defines a class
\[
    [B_{1}]\in KK^{\GG}(C_{0}(M),\CCl(V^{*}Q)).
\]
\end{prop}

\begin{proof}
The first thing we need to prove is that $B_{1}$ is self-adjoint 
and regular.  Observe first that $B_{1}$ is clearly symmetric.  
For each $[\phi]\in Q$, the localization $(X_{E^{1}})_{[\phi]}$ 
of $X_{E^{1}}$ in the sense of \cite{localglobal} and \cite{KaLe2} 
is just the 
finite dimensional Hilbert space
\[
    \HH_{[\phi]}:=\Lambda^{*}(V^{*}_{[\phi]}Q)\otimes\CB
\]
with the inner product coming from the Hermitian structure 
on $\Lambda^{*}(V^{*}_{[\phi]}Q)\otimes\mathbb{C}$, 
and the action of the localised operator $(B_{1})_{[\phi]}$ 
on $\HH_{[\phi]}$ is
\[
    (B_{1})_{[\phi]}\eta:=c_{L}(Z_{[\phi]})\eta.
\]
Since $(B_{1})_{[\phi]}$ is then self-adjoint on 
$\HH_{[\phi]}$, it follows from \cite[Th\'eor\`eme 1.18]{localglobal} 
that $B_{1}$ is self-adjoint and regular.

That $m(f)(1+B_{1}^2)^{-\frac{1}{2}}$ is a compact 
operator for all $f\in C_{0}(M)$ 
follows from the definition of Clifford multiplication.  Indeed, 
one has $c_{L}(Z_{[\phi]})^{2} = \|Z_{[\phi]}\|^{2} = h^{\sigma(\pi_{Q}([\phi]))}([\phi])^{2}$ since $dh^{\sigma(\pi_{Q}([\phi]))}_{[\phi]}$ has norm 1 for all $[\phi]$ as the dual of the tangent to the unique unit speed geodesic joining $\sigma(\pi_{Q}([\phi]))$ to $[\phi]$,
and so for any $f\in C_{0}(M)$, 
one simply has
\[
    (m(f)(1+B_{1}^{2})^{-\frac{1}{2}}\rho)([\phi]) = 
    \frac{f(\pi_{Q}([\phi]))}{(1+h^{\sigma(\pi_{Q}([\phi]))}([\phi])^{2})^{\frac{1}{2}}}\rho([\phi]).
\]
Since $f$ vanishes at infinity on the base $M$ of $Q\rightarrow M$, and since $[\phi]\mapsto (1+h^{\sigma(\pi_{Q}([\phi]))}([\phi])^{2})^{-\frac{1}{2}}$ vanishes at infinity on the fibres of $Q\rightarrow M$, the function $[\phi]\mapsto f(\pi_{Q}([\phi]))(1+h^{\sigma(\pi_{Q}([\phi]))}([\phi])^2)^{-\frac{1}{2}}$ is an element of $C_{0}(Q)$, so that $m(f)(1+B_{1}^{2})^{-\frac{1}{2}}$ is indeed a compact operator on the $\CCl(V^{*}Q)$-module $X_{E^{1}}$.

Concerning commutators, it is clear that $B_{1}$ commutes with the representation $m$ of $C_{0}(M)$.  Thus it only remains to prove that $B_{1}$ is appropriately equivariant.  The idea of this is essentially the unbounded version of analogous results by Connes \cite[Lemma 5.3]{cyctrans} and Kasparov \cite[Section 5.3]{KasparovEqvar}, but the details are somewhat technical so we give them here.  Fix $u\in \GG$ 
and $\rho\in\Gamma_{c}(Q_{r(u)};E^{1}|_{Q_{r(u)}})$.  We calculate
\begin{align*}
    (B_{1}-W^1_{u}B_{1}W^1_{u^{-1}})\rho([\phi]) =& c_{L}(Z_{[\phi]})\rho([\phi])-u_{*}(B_{1}W_{u^{-1}}^{1}\rho)(u^{-1}\cdot[\phi])\\ =& c_{L}(Z_{[\phi]})\rho([\phi]) - u_{*}(c_{L}(Z_{u^{-1}\cdot[\phi]})(W_{u^{-1}}^{1}\rho)(u^{-1}\cdot[\phi]))\\ =& c_{L}(Z_{[\phi]})\rho([\phi]) - u_{*}(c_{L}(Z_{u^{-1}\cdot[\phi]})(u^{-1}_{*}\rho([\phi])))\\ =& c_{L}(Z_{[\phi]}-u_{*} Z_{u^{-1}\cdot[\phi]})\rho([\phi])
\end{align*}
where on the third line we have used the identity \eqref{cL}.  Thus we must calculate a bound for the norm of the covector $Z_{[\phi]}-u_{*} Z_{u^{-1}\cdot [\phi]}$.

Denote  $\sigma_{r}:=\sigma(r(u))$ and $\sigma_{s}:=\sigma(s(u))$. With this notation, we have
\[
    Z_{[\phi]}-u_{*} Z_{u^{-1}\cdot [\phi]} = h^{\sigma_{r}}([\phi])dh^{\sigma_{r}}_{[\phi]} - u_{*} h^{\sigma_{s}}(u^{-1}\cdot [\phi])dh^{\sigma_{s}}_{u^{-1}\cdot [\phi]}.
\]
For any vector $\gamma\in V_{[\phi]}Q$ we have
\[
    (u_{*} dh^{\sigma_{s}}_{u^{-1}\cdot [\phi]})(\gamma) = dh^{\sigma_{s}}_{u^{-1}\cdot[\phi]}(u^{-1}_{*}\gamma) =  d(h^{\sigma_{s}}\circ u^{-1})_{[\phi]}(\gamma),
\]
giving $u_{*} dh^{\sigma_{s}}_{u^{-1}\cdot[\phi]} = d(h^{\sigma_{s}}\circ u^{-1})_{[\phi]}$, and since the action of $\GG$ is isometric on the fibres we get
\[
    (h^{\sigma_{s}}\circ u^{-1})([\phi]) = h(\sigma_{s},u^{-1}\cdot [\phi]) = h(u\cdot \sigma_{s},[\phi]) = h^{u\cdot \sigma_{s}}([\phi]).
\]
Thus
\[
    u_{*} dh^{\sigma_{s}}_{u^{-1}\cdot [\phi]} = dh^{u\cdot \sigma_{s}}_{[\phi]}.
\]
We then see that
\begin{align*}
    h^{\sigma_{r}}([\phi])dh^{\sigma_{r}}_{[\phi]}-u_{*} h^{\sigma_{s}}(u^{-1}\cdot[\phi])dh^{\sigma_{s}}_{u^{-1}\cdot [\phi]} =& h^{\sigma_{r}}([\phi])dh^{\sigma_{r}}_{[\phi]}-h^{u\cdot \sigma_{s}}([\phi])dh^{u\cdot \sigma_{s}}_{[\phi]}\\
    =& \frac{1}{2}d\bigg((h^{\sigma_{r}})^{2}-(h^{u\cdot \sigma_{s}})^{2}\bigg)_{[\phi]}\\
    =& \frac{1}{2}d\bigg((h^{\sigma_{r}}-h^{u\cdot \sigma_{s}})(h^{\sigma_{r}}+h^{u\cdot \sigma_{s}})\bigg)_{[\phi]}.
\end{align*}
By the argument \cite[Lemma 5.3]{KasparovEqvar}, we have
\[
    \|dh^{\sigma_{r}}_{[\phi]}-dh^{u\cdot \sigma_{s}}_{[\phi]}\|\leq 2h(\sigma_{r},u\cdot \sigma_{s})(h^{\sigma_{r}}([\phi])+h^{u\cdot \sigma_{s}}([\phi]))^{-1},
\]
which we use to estimate
\begin{align*}
    \|h^{\sigma_{r}}([\phi])dh^{\sigma_{r}}_{[\phi]}-u_{*} h^{\sigma_{s}}(u^{-1}\cdot [\phi])dh^{\sigma_{s}}_{u^{-1}\cdot [\phi]}\|^{2}\leq&\frac{1}{4}\|(dh^{\sigma_{r}}_{[\phi]}-dh^{u\cdot \sigma_{s}}_{[\phi]})(h^{\sigma_{r}}([\phi])+h^{u\cdot \sigma_{s}}([\phi]))\|^{2}\\&+\frac{1}{4}\|(h^{\sigma_{r}}([\phi])-h^{u\cdot \sigma_{s}}([\phi]))(dh^{\sigma_{r}}_{[\phi]}+dh^{u\cdot \sigma_{s}}_{[\phi]})\|^{2}\\
    \leq& h(\sigma_{r},u\cdot \sigma_{s})^{2}+(h(\sigma_{r},[\phi])-h(u\cdot \sigma_{s},[\phi]))^{2}\\
    =& h(\sigma_{r},u\cdot \sigma_{s})^{2}+h(\sigma_{r},[\phi])^{2}+h(u\cdot \sigma_{s},[\phi])^{2}\\&-2h(\sigma_{r},[\phi])h(u\cdot \sigma_{s},[\phi])\\
    \leq&2h(\sigma_{r},u\cdot \sigma_{s})^{2},
\end{align*}
where the last line is a consequence of the cosine inequality 
for spaces of non-positive sectional curvature \cite[Corollary 13.2]{helgason}.

Thus for all $[\phi]\in Q_{r(u)}$, we have 
$\|Z_{[\phi]}-u_{*} Z_{u^{-1}\cdot [\phi]}\|^{2}
\leq 2h(\sigma(r(u)),u\cdot \sigma(s(u)))^{2}$ independently of 
$[\phi]\in Q_{r(u)}$, implying that $B_{1}-W^{1}_{u}B_{1}W^{1}_{u^{-1}}$
extends to a bounded operator on $(X_{E^{1}})_{r(u)}$.  Moreover $u\mapsto h(\sigma(r(u)),u\cdot\sigma(s(u)))$ is continuous 
hence bounded on compact Hausdorff sets, so for any element $U_{i}$ of the cover $\mathcal{U} = \{U_{i}\}_{i\in I}$ of $\GG$ by Hausdorff open subsets, and for any $\varphi\in C_{c}(U_{i})$ and $f\in C_{0}(M)$ we have that
\[
    \varphi\cdot m_{i}^{r}(r_{i}^{*}(f))\cdot(r_{i}^{*}B_{1}
    -(W^{1})^{i}\circ s_{i}^{*}B_{1}\circ ((W^{1})^{i})^{-1})\in\LL(r_{i}^{*}X_{E^{1}}).
\]
It follows that $(C_{0}(M),{}_{m}X_{E^{1}},B_{1})$ is an 
unbounded equivariant Kasparov 
$C_{0}(M)$-$\CCl(V^{*}Q)$-module.
\end{proof}

\subsection{The foliation of the Connes fibration}

Before we can construct a second Kasparov module and the 
semifinite spectral triple associated to it, we need a closer 
study of the geometry and groupoid representation theory at our disposal.

\begin{defn}\label{bottcon}
	Assume $M$ to be equipped with a Riemannian metric $g$, with Levi-Civita connection $\nabla^{LC}$.  Then $N$ identifies with the subbundle $N = T\FF^{\perp}$ of $TM$, and we use the notation $X = X_{\FF}+X_{N}$ for the corresponding decomposition of vector fields into their normal and leafwise components.  Define a connection $\nabla^{\flat}$ on $N$ by the fomula
	\[
	\nabla^{\flat}_{X}(Y) = [X_{\FF},Y]_{N}+\nabla^{LC}_{X_{N}}(Y)_{N},\hspace{7mm}X\in\Gamma^{\infty}(M;TM),\,Y\in\Gamma^{\infty}(M;N).
	\]
	We refer to $\nabla^{\flat}$ as a \textbf{torsion-free Bott connection}.
\end{defn}

The terminology ``torsion-free" in Definition \ref{bottcon} refers to the fact that for any such connection $\nabla^{\flat}$ the associated torsion tensor
\[
T_{\nabla^{\flat}}(X,Y):=\nabla_{X}^\flat(Y_{N})-\nabla^\flat_{Y}(X_{N}) - [X,Y]_{N}
\]
vanishes for all $X,Y\in\Gamma^{\infty}(M;TM)$.  This fact follows from an easy calculation using the corresponding property of the Levi-Civita connection.  Bott connections more generally are characterised by the formula $\nabla^{\flat}_{X_{\FF}}(Y) = [X_{\FF},Y]_{N}$ for all smooth sections $Y$ of $N$ and $X_{\FF}$ of $T\FF$, and are the key ingredient for the Chern-Weil proof of Bott's vanishing theorem \cite{bott1}.  For us, the use of the Levi-Civita connection in Definition \ref{bottcon} serves a purpose that will become apparent in Section \ref{relconnes}.

Let us now come back to the frame bundle $\pi_{F}:F^{+}(N)\rightarrow M$.  
The total space of this bundle carries a foliation $\FF_{F}$ by the orbits of the action of $\GG$ on $F^{+}(N)$ defined in Equation \eqref{Delta}, whose normal we denote by $N_{F} := TF^{+}(N)/T\FF_{F}$.  The foliation $\FF_{F}$ is everywhere transverse to the fibres of $F^{+}(N)$, and its tangent bundle $T\FF_{F}$ projects fibrewise-isomorphically onto $T\FF$.  It is readily verified using a calculation in foliated coordinates that the connection form $\alpha^{\flat}\in\Omega^{1}(F^{+}(N);\mathfrak{gl}(q,\RB))$ associated to any Bott connection $\nabla^{\flat}$ on $N$ contains $T\FF_{F}$ in its kernel $HF^{+}(N):=\ker(\alpha^{\flat})$.  Consequently we find that the normal bundle to the foliation $\FF_{F}$ admits a decomposition
\begin{equation}
    N_{F} = VF^{+}(N)\oplus (HF^{+}(N)/T\FF_{F}).
    \label{eq:normal-split}
\end{equation}
The normal bundle $N_{F}$ is again a $\GG$-equivariant 
bundle, and with respect to the splitting \eqref{eq:normal-split} we write
\[
    u_*=\begin{pmatrix} \tilde{a}(u) & \tilde{c}(u) \\ 0 & \tilde{d}(u)\end{pmatrix}
\]
for the action of $u\in \GG$ on $N_{F}$. Note that the zero appearing in the bottom left corner is a consequence of the fact that by \eqref{Delta}, $\GG$ acts via diffeomorphisms between the fibres $GL^{+}(q,\RB)$ of $F^{+}(N)\rightarrow M$, and so preserves the bundle $VF^{+}(N)\rightarrow M$ of vectors tangent to the fibres.

Now we are not so interested in the frame bundle $F^{+}(N)$ as
the Connes fibration $Q$.  Since the action of $\GG$ on 
$F^{+}(N)$ commutes with the right action of $SO(q,\RB)$, however, 
we find that we also obtain a 
foliation on the total space of $\pi_{Q}:Q\rightarrow M$.

To be more specific, let $k:F^{+}(N)\rightarrow Q$ be the 
quotient map.  Then $T\FF_{Q}:=dk(T\FF_{F})$ is an 
integrable subbundle of $TQ$, which determines a 
foliation $\FF_{Q}$ of $Q$. Since $\pi_{Q}\circ k = \pi_{F}$, 
we see that $d\pi_{Q}$ maps $T\FF_{Q}$ isomorphically 
onto $T\FF$ making $\pi_{Q}:Q\rightarrow M$ a foliated bundle.  
The normal bundle $N_{Q}$ of $\FF_{Q}$ also admits a splitting
\[
    N_{Q} = VQ\oplus (HQ/T\FF_{Q}),
\]
where $HQ$ is the isomorphic image under $dk$ 
of the horizontal subbundle $HF^{+}(N)\subset TF^{+}(N)$.  
For convenience, we will denote $HQ/T\FF_{Q}$ by simply $H$.  
Thus,
\[
    N_{Q} = VQ\oplus H.
\]
Now, $d\pi_{Q}$ maps the fibres of $HQ$ isomorphically 
onto those of $TM$, and maps the fibres of $T\FF_{Q}$ 
isomorphically onto those of $T\FF$.  It follows that $d\pi_{Q}$ 
induces an isomorphism of the fibres of $H = HQ/T\FF_{Q}$ 
onto those of $N = TM/T\FF$.  We can then equip $H$ 
with a Euclidean metric in the following way, 
due to Connes \cite[Section 5]{cyctrans}.

\begin{prop}
\label{bundleofmetrics}
For $h_{1},h_{2}\in H_{[\phi]}$ and with $\cdot$ 
denoting the Euclidean inner product in $\RB^{q}$, the formula
\[
    m^{H}_{[\phi]}(h_{1},h_{2})
    :=\phi^{-1}(d\pi_{Q}(h_{1}))\cdot\phi^{-1}(d\pi_{Q}(h_{2}))
\]
determines a well-defined Euclidean metric on the 
bundle $H\rightarrow Q$.
\end{prop}

\begin{proof}
Suppose we were to choose a different 
representation $\phi' = \phi\circ A$ of $[\phi]$, 
where $A$ is some matrix in $SO(q,\RB)$.  
Then by the invariance of the Euclidean inner product under 
special orthogonal transformations we have
\begin{align*}
(\phi')^{-1}(d\pi_{Q}(h_{1}))\cdot(\phi')^{-1}(d\pi_{Q}(h_{2})) 
&= (A^{-1}\phi^{-1}(d\pi_{Q}(h_{1})))\cdot(A^{-1}\phi^{-1}(d\pi_{Q}(h_{2})))\\ 
&= \phi^{-1}(d\pi_{Q}(h_{1}))\cdot\phi^{-1}(d\pi_{Q}(h_{2})),
\end{align*}
giving well-definedness.  That we have defined a  metric 
follows from the linearity of the maps $\phi$ 
and $d\pi_{Q}$, and the fact that the Euclidean inner product is a 
metric on $\RB^{q}$.
\end{proof}

Remarkably, holonomy translations are orthogonal with respect to 
this Euclidean structure of $H$.  

\begin{prop}
\label{triangularstructure}
The normal bundle $N_{Q}\rightarrow Q$ of the 
foliation $\FF_{Q}$ of $Q$ is a $\GG$-equivariant vector 
bundle over the $\GG$-space $Q$.  Moreover, with respect 
to the splitting $N_{Q} = VQ\oplus H$, for $u\in \GG$ and $[\phi]\in Q_{s(u)}$ the 
holonomy action $u_{*}:(N_{Q})_{[\phi]}\rightarrow (N_{Q})_{u\cdot[\phi]}$ 
has the form
\begin{equation}\label{matrix}
    u_* = \left(\begin{array}{cc} a(u) & c(u) \\ 0 & d(u)\end{array}\right),
\end{equation}
with $a(u):V_{[\phi]}Q\rightarrow V_{u\cdot[\phi]}Q$ and 
$d(u):H_{[\phi]}\rightarrow H_{u\cdot[\phi]}$ orthogonal and 
orientation-preserving.
\end{prop}

\begin{proof}
The holonomy groupoid for the foliation $\FF_{Q}$ of $Q$ 
is precisely the groupoid $Q\rtimes \GG$, under which the 
normal bundle $N_{Q}\rightarrow Q$ is therefore equivariant.  
Thus $N_{Q}\rightarrow Q$ is a $\GG$-equivariant vector bundle over the $\GG$-space $Q$.

Proposition \ref{verticalequi} tells us that 
$a(u):V_{[\phi]}Q\rightarrow V_{u\cdot[\phi]}Q$ is 
orthogonal and orientation-preserving, and that the 
vertical bundle is preserved under holonomy translation, 
which accounts for the 0 appearing in the bottom left 
corner of \eqref{matrix}.  Since 
$\pi_{Q}:Q\rightarrow M$ is the anchor map for the 
$\GG$-space $Q$ it is $\GG$-equivariant, implying that the 
identification $d\pi_{Q}$ of fibres of $H$ with those of $N$ is 
also $\GG$-equivariant.  

That $d(u):H_{[\phi]}\rightarrow H_{u\cdot[\phi]}$ 
is orientation-preserving is then a consequence of the fact that $d(u)$ 
may be identified with the orientation-preserving action of $u$ 
on the fibres of $N$.  That $d(u)$ is orthogonal is a consequence 
of the following calculation for $h_{1},h_{2}\in H_{[\phi]}$:
\begin{align*}
    m^{H}_{u\cdot[\phi]}(d(u)h_{1},d(u)h_{2}) 
    =& (u_{*}\circ\phi)^{-1}((d\pi_{Q}\circ d(u))(h_{1}))\cdot(u_{*}\circ\phi)^{-1}((d\pi_{Q}\circ d(u))(h_{1}))\\ 
    =& (\phi^{-1}\circ u_{*}^{-1})((u_{*}\circ d\pi_{Q})(h_{1}))\cdot(\phi^{-1}\circ u_{*}^{-1})((u_{*}\circ d\pi_{Q})(h_{2}))\\ 
    =& \phi^{-1}(d\pi_{Q}(h_{1}))\cdot\phi^{-1}(d\pi_{Q}(h_{2})) 
    = m^{H}_{[\phi]}(h_{1},h_{2}),
\end{align*}
where on the second line, we have used the equivariance of the 
anchor map $d\pi_{Q}$ between $H$ and $N$.
\end{proof}

The triangular shape of the matrix in 
Proposition \ref{triangularstructure} is what is 
referred to as an \emph{almost isometric} 
or \emph{triangular structure} by Connes \cite{cyctrans} and Connes-Moscovici \cite{CM} respectively.

The map $c(u):H_{[\phi]}\rightarrow V_{u\cdot[\phi]}Q$, for $u\in \GG$ and $[\phi]\in Q_{s(u)}$, is where 
the interesting representation theory is encoded.  Currently, 
however, the range of $c(u)$ is too high in dimension to 
be of much use, and these extra dimensions need to be ``traced out".  
Observing that there is indeed a canonical trace 
$\tr_{F^{+}(N)}:VF^{+}(N)\rightarrow \RB$ induced fibrewise by 
the usual matrix trace on $\mathfrak{gl}(q,\RB) = M_{q}(\RB)$, 
we now check that we can apply this map to $VQ$ also.

\begin{lemma}
\label{trace}
The map $\tr_{F^{+}(N)}:VF^{+}(N)\rightarrow\RB$ 
descends to a well-defined map $\tr_{Q}:VQ\rightarrow\RB$ 
for which $\tr_{Q}\circ a(u) = \tr_{Q}$ for all $u\in \GG$.
\end{lemma}

\begin{proof}
For $A\in GL^{+}(q,\RB)$, we denote by 
$R_{A}:F^{+}(N)\rightarrow F^{+}(N)$ the map 
$\phi\mapsto \phi\cdot A$.  By definition, 
the action of $A\in SO(q,\RB)$ on $VF^{+}(N)$ is then given for $\phi\in F^{+}(N)$ and $v_{\phi}\in V_{\phi}F^{+}(N)$ by
\[
    v_{\phi}\cdot A:=(dR_{A})_{\phi}(v_{\phi}).
\]
We compute
\[
    (dR_{A})_{\phi}(v_{\phi}) 
    = \frac{d}{dt}(\phi\cdot\exp(tv)\cdot A)\bigg|_{t=0} 
    = \frac{d}{dt}((\phi\cdot A)\cdot(A^{-1}\exp(tv)A))\bigg|_{t=0} 
    = (A^{-1}vA)_{\phi\cdot A},
\]
from which we deduce that the action of $A\in SO(q,\RB)$ 
in the trivialisation $VF^{+}(N) =F^{+}(N)\times\mathfrak{gl}(q,\RB)$ is given by
\[
    (\phi,v)\cdot A = (\phi\cdot A,A^{-1}vA)
\]
for all $\phi\in F^{+}(N)$, $v\in\mathfrak{gl}(q,\RB)$.  Now, $\tr_{F^{+}(N)}:F^{+}(N)\times\mathfrak{gl}(q,\RB)\rightarrow \RB$ 
is by definition
\[
    \tr_{F^{+}(N)}(\phi,v):=\tr(v),
\]
with $\tr$ denoting the usual matrix trace on $q\times q$ matrices, 
and with the range $\RB$ of $\tr_{F^{+}(N)}$ carrying the trivial 
action of $SO(q,\RB)$.  Then since the matrix trace is 
invariant under conjugation, we see that $\tr_{F^{+}(N)}$ is equivariant:
\[
    \tr_{F^{+}(N)}((\phi,v)\cdot A) = \tr(A^{-1}vA) = \tr(v) 
    = \tr_{F^{+}(N)}(\phi,v)\cdot A,
\]
and so descends to a well-defined 
map $\tr_{Q}:VQ\rightarrow\RB$.

For the second assertion, note that since $u$ commutes 
with the quotient map $Q:F^+N\to Q$ and since $u_{*}$ acts as the identity on the fibres 
of $VF^{+}(N) = F^{+}(N)\times\RB^{q^2}$ by \eqref{uonv}, we have
\[
    \tr_{Q}\circ a(u)\circ dQ = \tr_{Q}\circ dQ\circ \id = \tr_{Q}\circ dQ.
\]
Since $dQ$ is surjective, we conclude that
\[
    \tr_{Q}\circ a(u) = \tr_{Q}
\]
as claimed.
\end{proof}

\begin{rmk}\normalfont
	Note that what makes Lemma \ref{trace} possible is the fact that the map $v\mapsto \tr(v)$ on $\mathfrak{gl}(q,\RB)$ is invariant under conjugation by invertible matrices.  Thus in fact we could replace $\tr$ with any other invariant polynomial on $\mathfrak{gl}(q,\RB)$, parallelling the Chern-Weil construction of characteristic classes, and still obtain a well-defined (but no longer necessarily linear) map on the vertical tangent bundle of the Connes fibration.  This observation is due to M. T. Benameur.
\end{rmk}

Let us put Lemma \ref{trace} to use in simplifying the groupoid representation theory.  For $u\in \GG$ and $[\phi]\in Q_{s(u)}$, define
\[
    \delta(u):=\tr_{Q}\circ c(u):H_{[\phi]}\rightarrow\RB.
\]
This $\delta(u)$ is linear, and so can be regarded as an 
element of $H_{[\phi]}^{*}$.  We also define
\[
    \theta(u):=d(u^{-1})^{t}:H_{[\phi]}^{*}\rightarrow H_{u\cdot[\phi]}^{*},
\]
the action on the covector bundle for $H$ coming from the 
transpose of $d(u^{-1}):H_{u\cdot[\phi]}\rightarrow H_{[\phi]}$.  We have the following ``$ax+b$ group"-type transformation laws.

\begin{lemma}
\label{identities}
For all $u,v\in \GG^{(2)}$, we have
\[
    \theta(uv) = \theta(u)\theta(v),
\quad\mbox{and}\quad
    \delta(uv) = \delta(v)+\theta(v^{-1})\delta(u).
\]
\end{lemma}

\begin{proof}
These identities follow from the triangular structure of the 
matrices \eqref{matrix} and Lemma \ref{trace}.  Specifically, 
since $\GG$ acts on $N_{Q}$ we have
\[
    \left(\begin{array}{cc} a(uv) & c(uv) \\ 0 & d(uv)\end{array}\right) 
    = \left(\begin{array}{cc} a(u) & c(u) \\ 0 & d(u)\end{array}\right)\left(\begin{array}{cc} a(v) & c(v) \\ 0 & d(v)\end{array}\right) 
    = \left(\begin{array}{cc} a(u)a(v) & a(u)c(v) + c(u)d(v) \\ 0 & d(u)d(v)\end{array}\right),
\]
from which we immediately deduce that $d(uv) = d(u)d(v)$ 
and hence $\theta(uv) = \theta(u)\theta(v)$.  We also calculate
\begin{align*}
    \delta(uv) =& \tr_{Q}\circ c(uv) = \tr_{Q}\circ a(u)\circ c(v)+\tr_{Q}\circ c(u)\circ d(v)\\ =& \tr_{Q}\circ c(v) + \tr_{Q}\circ c(u)\circ d(v) = \delta(v) + \theta(v^{-1})\delta(u),
\end{align*}
using Lemma \ref{trace} for the third equality, giving the desired identities.
\end{proof}

\subsection{The Vey Kasparov module}

We now go about constructing a second Kasparov module, referred to in this paper as the Vey Kasparov module since it appears to be analogous to the Vey homomorphism considered in previous work \cite{hurd, Duminy}.  Our first job in constructing a second Kasparov module is to endow the total space $H^{*}$ of the horizontal covector bundle $\pi_{H^{*}}:H^{*}\rightarrow Q$ with an action of $\GG$ that encodes both $\theta$ and $\delta$ from Lemma \ref{identities}.

\begin{prop}\label{HGspace}
	For $u\in \GG$ and $\eta\in H^{*}|_{Q_{s(u)}}$, the formula
	\[
	u\cdot\eta:=\theta(u)\eta+\delta(u^{-1})
	\]
	determines the structure of a $\GG$-space on $H^{*}$ with anchor map $\pi_{Q}\circ\pi_{H^{*}}:H^{*}\rightarrow M$.
\end{prop}

\begin{proof}
	It is clear that $(\pi_{Q}\circ\pi_{H^{*}})(u\cdot\eta) = r(u)$ for all $u\in \GG$ and $\eta\in H^{*}|_{Q_{s(u)}}$, and since by Lemma \ref{identities} $\theta$ is the identity on units and $\delta$ is zero on units we get $(\pi_{Q}\circ\pi_{H^{*}})(\eta)\cdot\eta = \eta$ for all $\eta$.  Thus it remains only to check that $(uv)\cdot\eta= u\cdot (v\cdot\eta)$ for all $(u,v)\in \GG^{(2)}$ and $\eta\in H^{*}|_{Q_{s(v)}}$.  For this we simply have
	\begin{align*}
	(uv)\cdot\eta =& \theta(uv)\eta+\delta(v^{-1}u^{-1}) = \theta(u)\big(\theta(v)\eta+\delta(v^{-1})\big)+\delta(u^{-1})= u\cdot(v\cdot \eta),
	\end{align*}
	with the second equality being a consequence of Lemma \ref{identities}.
\end{proof}

We can now construct another dual Dirac class in much 
the same way as we did for the Connes fibration.  
Consider the bundle $VH^{*}:=\ker(d\pi_{H^{*}})$ of 
vertical tangent vectors over the horizontal covector bundle 
$\pi_{H^{*}}:H^{*}\rightarrow Q$, and denote by $\pi_{H}:H\rightarrow Q$ the projection for the horizontal bundle.  Since the fibres of $H^{*}$ are vector spaces, we have $V_{\eta}H^{*}_{[\phi]}\cong H^{*}_{[\phi]}$ for all $[\phi]\in Q$ and $\eta\in H^{*}_{[\phi]}$.  Thus the dual space $V^{*}_{\eta}H^{*}_{[\phi]}$ is a copy of $H_{[\phi]}$ and so we can write $V^{*}H^{*}$ as the fibered product
\[
	V^{*}H^{*}\cong H^{*}\times_{\pi_{H^{*}},\pi_{H}}H,
\]
regarded as a vector bundle over $H^{*}$ by using the projection onto the first factor.  Since $H$ is a $\GG$-equivariant Euclidean bundle over $Q$ via the map $d$ in Proposition \ref{triangularstructure}, for all $u\in \GG$, $\eta\in H^{*}|_{Q_{s(u)}}$ and $h\in H|_{Q_{s(u)}}$, the formula
\[
u_{*}(\eta,h):=(u\cdot\eta,d(u)h) = (\theta(u)\eta+\delta(u^{-1}),d(u)h)
\]
defines on $V^{*}H^{*}$ the structure of a $\GG$-equivariant Euclidean bundle over the $\GG$-space $H^{*}$.  Then by functoriality $\Cliff(V^{*}H^{*})$ is a $\GG$-equivariant bundle over $H^{*}$, and we denote the 
action of $u\in \GG$ on $k\in \Cliff(V^{*}H^{*}|_{H^{*}_{[\phi]}})$ 
by $k\mapsto u_{\diamond} k$ for all $[\phi]\in Q_{s(u)}$.  Using these facts together with Proposition \ref{HGspace}, the following result is clear.

\begin{prop}
	The formula
	\[
	\alpha^{2}_{u}(\zeta)(\eta):=u_{\diamond}\zeta(u^{-1}\cdot\eta) = u_{\diamond}\zeta\big(\theta(u^{-1})\eta+\delta(u)\big)
	\]
	defined for $\zeta\in \CCl(V^{*}H^{*})$, $u\in \GG$ and $\eta\in H^{*}_{[\phi]}$ with $[\phi]\in Q_{r(u)}$, determines the structure of a $\GG$-algebra on $\CCl(V^{*}H^{*})$.
\end{prop}
We now come to the definition of an appropriate Hilbert module.  Let
\[
    E^{2}:=\Lambda^{*}(V^{*}H^{*})\otimes\CB
\]
be the complexified exterior algebra bundle of $V^{*}H^{*}$ over $H^{*}$, 
and define
\[
    X_{E^{2}}:=\Gamma_{0}(H^{*};E^{2}),
\]
which is a Hilbert $\CCl(V^{*}H^{*})$-module whose structure as such is determined in the same way as for $X_{E^{1}}$ using the identification of $E^{2}$ with $\Cliff(V^{*}H^{*})$ as vector bundles.

By equivariance of $V^{*}H^{*}$ over $H^{*}$ and functoriality, for $u\in \GG$, $[\phi]\in Q_{s(u)}$ and $\eta\in H^{*}_{[\phi]}$ 
we obtain a unitary holonomy transport map 
$u_{*}:E^{2}_{\eta}\rightarrow E^{2}_{u\cdot \eta}$ and an isomorphism 
$W^{2}_{u}:\Gamma_{0}(H^{*}_{[\phi]};E^{2}|_{H^{*}_{[\phi]}})
\rightarrow\Gamma_{0}(H^{*}_{u\cdot[\phi]};E^{2}|_{H^{*}_{u\cdot[\phi]}})$ 
of Banach spaces defined by
\[
    (W^{2}_{u}\zeta)(\eta):=u_{*}\zeta(u^{-1}\cdot\eta) = u_{*}\zeta\big(\theta(u^{-1})\eta+\delta(u)\big).
\]
Using Lemma \ref{cliffacthom}, we observe that
\begin{align*}
\langle W^{2}_{u}\zeta_{1},W^{2}_{u}\zeta_{2}\rangle_{\CCl(V^{*}H^{*})_{r(u)}}(\eta) 
	&= u_{\diamond}\big\langle \zeta_{1}\big(\theta(u^{-1})\eta+\delta(u)\big),\zeta_{2}\big(\theta(u^{-1})\eta+\delta(u)\big)\big\rangle\\
	&= \alpha^{2}_{u}(\langle\zeta_{1},\zeta_{2}\rangle_{\CCl(V^{*}H^{*})_{s(u)}})(\eta)
	\end{align*}
for all $u\in \GG$, $[\phi]\in Q_{r(u)}$ and $\eta\in H^{*}_{[\phi]}$, 
so $(X_{E^{2}},W^{2})$ is a $\GG$-Hilbert $\CCl(V^{*}H^{*})$-module.

We can define an unbounded operator $B_{2}$ on the dense 
submodule $X^{c}_{E^{2}} = \Gamma_{c}(H^{*};E^{2})$ of $X_{E^{2}}$ 
by the formula
\[
    (B_{2}\zeta)(\eta):=c_{L}(\eta)\zeta(\eta),
\]
where for $c_{L}(\eta)$ we regard $\eta\in H^{*}$ as a vertical covector in $V^{*}H^{*} = H^{*}\times_{\pi_{H^{*}},\pi_{H}}H$ using the Euclidean metric on $H$.

Finally, we take $m^{2}$ to be the representation of $C_{0}(Q)$ on $X_{E^{2}}$ defined by
\[
    m^{2}(f)\zeta(\eta):=f(\pi_{H^{*}}(\eta))\zeta(\eta).
\]
Using the fact that $\pi_{H^{*}}$ is an equivariant map and that $\pi_{H^{*}}(\eta+\eta') = \pi_{H^{*}}(\eta) = [\phi]$ for all $[\phi]\in Q$ and  $\eta,\eta'\in H^{*}_{[\phi]}$, a routine calculation shows that $m^{2}$ is an equivariant representation.

\begin{prop}
\label{B2}
The triple $(C_{0}(Q),{}_{m^{2}}X_{E^{2}},B_{2})$ is an 
unbounded $\GG$-equivariant Kasparov 
$C_{0}(Q)$-$\CCl(V^{*}H^{*})$-module, defining a class
\[
    [B_{2}]\in KK^{\GG}(C_{0}(Q),\CCl(V^{*}H^{*})).
\]
\end{prop}

\begin{proof}
The proof is essentially the same as the proof of Proposition \ref{B1}.  
The only part that must be changed is checking the equivariance condition.  
For any $u\in \GG$, $[\phi]\in Q_{r(u)}$ and $\eta\in H^{*}_{[\phi]}$, we have
\begin{align*}
    (W^{2}_{u}B_{2}W^{2}_{u^{-1}})\zeta(\eta) 
    =&  u_{*}(B_{2}W^{2}_{u^{-1}}\zeta)\big(\theta(u^{-1})\eta+\delta(u)\big) \\ 
    =& u_{*}\Big(c_{L}\big(\theta(u^{-1})\eta+\delta(u)\big)(W^{2}_{u^{-1}}\zeta)\big(\theta(u^{-1})\eta+\delta(u)\big)\Big) \\ 
    =&u_{*}\Big( c_{L}\big(\theta(u^{-1})\eta+\delta(u)\big)\Big(u^{-1}_{*}
    \zeta\big(\theta(u)\big(\theta(u^{-1})\eta+\delta(u)\big)+\delta(u^{-1})\big)\Big)\Big)\\
    =&u_{*} \Big(c_{L}\big(\theta(u^{-1})\eta+\delta(u)\big)\big(u^{-1}_{*}
    \zeta(\eta)\big)\Big)\\
    =& c_{L}\big(\eta-\delta(u^{-1})\big)\zeta(\eta)
\end{align*}
where the last line follows from the identity $\theta(u)\delta(u) = -\delta(u^{-1})$
arising from Lemma \ref{identities}, 
together with the identity \eqref{cL}.  We then have
\[
    B_{2}-W^{2}_{u}B_{2}W^{2}_{u^{-1}} = c_{L}(\delta(u^{-1})),
\]
which defines a bounded operator on $(X_{E^{2}})_{r(u)}$.  
The rest of the proof is then the same as in Proposition \ref{B1}.
\end{proof}

\section{The index theorem}

\subsection{Some simplifications in codimension 1}\label{subsec:simp}

There are important simplifications in the codimension 1 case.  
Observe that for a codimension 1, transversely orientable 
foliation $\FF$ of $M$, the conormal bundle $N^{*}\rightarrow M$ 
is trivialised by a choice of orientation, which is given by a choice
of a transverse volume 
form $\omega$. 
  Such a choice 
determines a dual section $\omega^{*}$ of $N\rightarrow M$ and 
hence a map $t:N\rightarrow\RB$ defined by the 
equality $n = t(n)\omega^{*}$ for $n\in N$.  Thus
\[
    N = M\times \RB.
\]
The action of $u\in \GG$ on $N$ will then be denoted by
\begin{equation}
    u_{*}(s(u),n):= (r(u),\Delta(u)n),
\label{eq:action-codim1}
\end{equation}
with $\Delta:\GG\rightarrow\RB^{*}_{+}$ a multiplicative 
homomorphism.  Observe that under the 
correspondence $\omega\mapsto \omega^{*}$, 
this $\Delta(u)$ is precisely the Radon-Nikodym 
derivative of the transverse volume form $\omega$ 
with respect to the holonomy translation $u$.  The 
principal $\RB^{*}_{+}$-bundle $F^{+}(N)$ of positively 
oriented frames for $N$, which coincides with the 
Connes fibration $Q$ since $SO(1,\RB) = 1$, 
is then also trivial under the map $\phi\mapsto(\pi_{Q}(\phi),t\circ\phi)$:
\[
    Q = M\times\RB^{*}_{+}.
\]
The action of $u$ on the fibres of 
$Q$, defined by \eqref{Delta} since $q =1$, is induced by the same homomorphism $\Delta(u)$:
\[
    u\cdot(s(u),b):=(r(u),\Delta(u)b).
\]
We will assume for ease of calculation that
\[
    Q = M\times\RB
\]
using the logarithm map on the fibres, so that the action of a groupoid element $u\in \GG$ on $Q$ is now given by
\[
    u\cdot(s(u),c) = (r(u),c+\log\Delta(u)).
\]

The horizontal and vertical bundles are both trivial line bundles, so
\[
    N_{Q} = VQ\oplus H = Q\times(\RB\oplus\RB).
\]
Here we regard the horizontal bundle $H = Q\times\RB$ as a Euclidean bundle with metric $m$ arising from $Q$ defined as in Proposition \ref{bundleofmetrics} by
\[
    m_{(x,c)}^{H}(h_{1},h_{2}):=(e^{-c}h_{1})\cdot(e^{-c}h_{2}) = e^{-2c}h_{1}h_{2}.
\]
We use the metric $m^{H}$ to identify $H$ 
with its dual $H^{*}$, by mapping $h\in H$ 
to the functional $m^{H}(h,\cdot)$.  More 
precisely, we identify $h\in H_{(x,c)} =\RB$ 
with $\eta_{h}:=e^{-2c}h\in H^*_{(x,c)}$.  We then 
find that the resulting metric on $H^{*}$ is
\[
	m^{H^{*}}_{(x,c)}(\eta_{h},\eta_{h'}):=m^{H}_{(x,c)}(h,h') 
	= e^{-2c}hh' = e^{2c}\eta_{h}\eta_{h'}.
\]
Under this identification, the map 
$\theta(u):H^{*}_{(s(u),c)}\rightarrow H^{*}_{(r(u),c+\log\Delta(u))}$ 
is precisely $\eta\mapsto \Delta(u^{-1})\eta$.

With no need to trace over the vertical fibres in the 
codimension 1 case, we can then write the triangular structure of a holonomy 
transformation $u\in \GG$ as
\[
    u_* = \left(\begin{array}{cc} 1 & \delta(u) \\ 0 & \Delta(u)\end{array}\right).
\]
This action of $u_*$ on $VQ\oplus H\subset TQ$ is the 
differential of the action of $u$ on $Q$.  It follows then 
that $\delta(u)$ is the derivative with respect to the 
transverse coordinate in $M$ of the map $c\mapsto c+\log\Delta(u)$ 
on the fibres of $Q$.  Since the normal bundle $N$ over $M$ has been trivialised, we can write this derivative as the scalar $\delta(u) = \partial\log\Delta(u)$, with $\partial$ denoting the derivative with respect to the transverse coordinate.  Thus
\[
    u_{*} = \left(\begin{array}{cc} 1 & \partial\log\Delta(u) \\ 0 & \Delta(u)\end{array}\right).
\]

Let us now consider the Kasparov module $[B_{2}]$.  
The right-hand algebra in this case is $\CCl(V^{*}H^{*})$, 
and since for each $(x,c,\eta)\in H^{*}$ we can identify vertical tangent vectors in $V_{(x,c,\eta)}H^{*}$ with vectors in $H^{*}_{(x,c)}$, it follows that we can identify vertical covectors in $V^{*}_{(x,c,\eta)}H^{*}$ with linear functionals $H^{*}_{(x,c)}\rightarrow\RB$.  Observe then 
that there is a nonvanishing section $\kappa$ of $V^{*}H^{*}\rightarrow H^{*}$ defined by
\[
	\kappa(x,c,\eta)
	:=e^{c}\eta,\quad\mbox{for}\quad (x,c,\eta)\in H^{*}.
\]
One has
\[
	\kappa(r(u),c+\log\Delta(u),\Delta(u^{-1})\eta) = e^{c+\log\Delta(u)}\Delta(u^{-1})\eta = e^{c}\eta = \kappa(s(u),c,\eta),
\]
so $\kappa$ is invariant under the action of $\GG$ and therefore defines a trivialisation $V^{*}H^{*}\cong H^{*}\times\RB$ for which the action of $\GG$ is given by
\[
	u_*(s(u),c,\eta,s) = (r(u),c+\log\Delta(u),\Delta(u^{-1})\eta,s)
	\quad\mbox{for}\quad c\in Q,\ \ s\in\RB,\ \ \eta\in H^{*}_{(s(u),c)}.
\]
It follows that we can take $\CCl(V^{*}H^{*})$ to be $C_{0}(H^{*})\otimes\Cliff(\RB)$, where $\GG$ acts trivially on $\Cliff(\RB)$. That is, for all 
$f\otimes e\in C_{0}(H^{*})\otimes\Cliff(\RB)$ we have
\[
	\alpha^{2}_{u}(f\otimes e)(r(u),c,\eta) = f(s(u),c-\log\Delta(u),\Delta(u)\eta+\partial\log\Delta(u))\otimes e,\quad \eta\in H^{*}_{(r(u),c)}.
\]
We define therefore an action $\alpha$ of $\GG$ on $C_{0}(H^{*})$ by
\[
\alpha_{u}(f)(r(u),c,\eta):=f(s(u),c-\log\Delta(u),\Delta(u)\eta+\partial\log\Delta(u))
\]
for $f\in C_{0}(H^{*})$, so that $\alpha^{2}_{u}(f\otimes e) = \alpha_{u}(f)\otimes e$ for all $u\in \GG$ and $e\in\Cliff(\RB)$.

The same remarks carry over to the exterior bundle $\Lambda^{*}V^{*}H^{*}$, so that $\Gamma_{0}(H^{*};\Lambda^{*}(V^{*}H^{*})\otimes\CB)$ is just $C_{0}(H^{*})\otimes\Cliff(\RB)$, on which the representation $W^{2}$ of $\GG$ is defined by the same formula as $\alpha^{2}$:
\[
	W^{2}_{u}(\rho\otimes e)(r(u),c,\eta) = \rho(s(u),c-\log\Delta(u),\Delta(u)\eta+\partial\log\Delta(u))\otimes e
\]
for all $\rho\otimes e\in C_{0}(H^{*})\otimes\Cliff(\RB)$.  We thus define an action $W$ of $\GG$ on the Hilbert $C_{0}(H^{*})$-module $C_{0}(H^{*})$ by
\[
W_{u}(\rho)(r(u),c,\eta):=\rho(s(u),c-\log\Delta(u),\Delta(u)\eta+\partial\log\Delta(u))
\]
for all $\rho\in C_{0}(H^{*})$, and we see that $W^{2}_{u}(\rho\otimes e) = W_{u}(\rho)\otimes e$ for all $u\in \GG$ and $e\in\Cliff(\RB)$.

Finally, the operator $B_{2}$ acts on $C_{0}(H^{*})\otimes\Cliff(\RB)$ by
\[
    (B_{2}\rho\otimes e)(x,c,\eta):=e^{c}\eta\rho(x,c,\eta)\otimes c_L(e_{1})e,\quad e\in \Cliff(\RB),\ \eta\in H^*_{(x,c)},
\]
where $c_L$ is the left Clifford multiplication and $e_{1}$ is a fixed element of $\Cliff(\RB)$ with square 1.  We can now proceed with the construction of a spectral triple from this data and the proof of the index theorem relating the spectral triple to the Godbillon-Vey invariant.

\subsection{The spectral triple}
\label{subsec:spec-trip}

Applying the descent map to the equivariant 
Kasparov module $(C_{0}(Q),{}_{m^{2}}X_{E^{2}},B_{2})$ 
of Proposition \ref{B2} in codimension 1 gives us by Proposition \ref{descent} a 
Kasparov module
\begin{equation}
\label{kasmod}
	(\Gamma_{c}(Q\rtimes \GG,\Omega^{\frac{1}{2}}), X_{E^{2}}\rtimes_{r} \GG, r^{*}B_{2})
\end{equation}
which defines a class in $KK(C_{0}(Q)\rtimes_{r} \GG, \CCl(V^{*}H^{*})\rtimes_{r} \GG)$.  
By the remarks of the previous section, we actually have
\[
	\CCl(V^{*}H^{*})\rtimes_{r} \GG 
	= (C_{0}(H^{*})\otimes\Cliff(\RB))\rtimes_{r} \GG 
	= (C_{0}(H^{*})\rtimes_{r} \GG)\otimes\Cliff(\RB)
\]
since $\GG$ acts trivially on $\Cliff(\RB)$.  Thus the 
module \eqref{kasmod} can be 
replaced \cite[Proposition 13, Appendix A, Chapter 4]{ncg} by 
the odd Kasparov $C_{0}(Q)\rtimes_{r} \GG$-$C_{0}(H^{*})\rtimes_{r} \GG$-module
\begin{equation}
\label{oddkasmod}
	(\Gamma_{c}(Q\rtimes \GG,\Omega^{\frac{1}{2}}), C_{0}(H^{*})\rtimes_{r} \GG, \BB)
\end{equation}
where we define $\BB$ on $\Gamma_{c}(H^{*}\rtimes \GG;\Omega^{\frac{1}{2}})\subset C_{0}(H^{*})\rtimes_{r} \GG$ by
\[
    (\BB\rho)_{u}(x,c,\eta):=(\BB_{r(u)}\rho_{u})(x,c,\eta):=e^{c}\eta\rho_{u}(x,c,\eta),\quad \eta\in\RB.
\]
Here we are using density of $\Gamma_{c}(H^{*}\rtimes \GG;\Omega^{\frac{1}{2}})$ in $C_{0}(H^{*})\rtimes_{r} \GG$ and density of $\Gamma_{c}(Q\rtimes \GG;\Omega^{\frac{1}{2}})$ in $C_{0}(Q)\rtimes_{r} \GG$ as in the final paragraph of Section 2.3.

The $\GG$-invariant transverse volume forms of interest on $Q$ and $H^{*}$ respectively are
\begin{equation}\label{volforms}
d\nu_{Q} = e^{-c}\omega\wedge dc,\hspace{7mm}d\nu_{H^{*}} = \omega\wedge dc\wedge d\eta,
\end{equation} 
and we let $\tau_{Q}$ and $\tau_{H^{*}}$ be the corresponding traces on $\Gamma_{c}(Q\rtimes\GG;\Omega^{\frac{1}{2}})$ and $\Gamma_{c}(H^{*}\rtimes\GG;\Omega^{\frac{1}{2}})$ defined by integration against $d\nu_{Q}$ and $d\nu_{H^{*}}$.

Putting the trace $\tau_{H^{*}}$ together with the odd Kasparov 
module \eqref{oddkasmod}, by Proposition \ref{sem1} we 
obtain an odd semifinite spectral triple
\[
	(\AA,\HH,\BB)
\]
relative to $(\NN,\tau)$ where:
\begin{enumerate}
	\item $\AA = \Gamma_{c}(Q\rtimes \GG;\Omega^{\frac{1}{2}})$ acts by convolution operators on
	\item $\HH$, the Hilbert space completion of $\Gamma_{c}(H^{*}\rtimes \GG;\Omega^{\frac{1}{2}})$ in the inner product
	\[
		(\rho_{1}|\rho_{2}) = \tau_{H^{*}}(\rho_{1}^**\rho_{2}),
	\]
	\item $\BB$ is regarded as an operator on $\HH$ with 
	domain $\Gamma_{c}(H^{*}\rtimes \GG;\Omega^{\frac{1}{2}})$,
	\item $\NN$ is the weak closure of 
	$\Gamma_{c}(H^{*}\rtimes \GG;\Omega^{\frac{1}{2}})$ 
	in the bounded operators on $\HH$ and,
	\item $\tau$ is the normal extension of $\tau_{H^*}$ to $\NN$.
\end{enumerate}

We now apply the semifinite local index formula 
to $(\AA,\HH,\BB)$ to prove the codimension 1 Godbillon-Vey index theorem.

\subsection{The index theorem}

We will apply the residue cocycle of \cite[Definition 3.2]{cgrs2} to prove the following theorem.

\begin{thm}
\label{mainthm}
Let $(M,\FF)$ be a foliated manifold of codimension 1.  
The Chern character of the semifinite spectral triple 
$(\AA,\HH,\BB)$ given in Section \ref{subsec:spec-trip} 
is, up to a factor of $(2\pi i)^{\frac{1}{2}}$, the global, non-\'{e}tale analogue of Godbillon-Vey cyclic cocycle of Connes 
and Moscovici \cite[Proposition 19]{backindgeom}.
\end{thm}

To apply the local index formula of \cite{cgrs2} we need to 
check the summability and smoothness of the spectral triple.

\begin{lemma}
The spectral triple $(\AA,\HH,\BB)$ is smoothly 
summable of spectral dimension $p = 1$ and has isolated spectral dimension.
\end{lemma}

\begin{proof}
We first check finite summability.  For $s\in\RB$,
$a\in \Gamma_{c}(Q\rtimes \GG;\Omega^{\frac{1}{2}})$ and $\rho\in \HH$, we calculate
\begin{align*}
	(a(1+&\BB^{2})^{-\frac{s}{2}}\rho)_{u}(x,c,\eta) 
	= \int_{v\in \GG^{r(u)}} a_{v}(x,c)\big(W_{v}(1+\BB_{s(v)}^{2})^{-\frac{s}{2}}\rho_{v^{-1}u})(x,c,\eta)\\ 
	=& \int_{v\in \GG^{r(u)}} a_{v}(x,c)\big(1+e^{2(c-\log\Delta(v))}(\Delta(v)\eta+\partial\log\Delta(v))^2\big)^{-\frac{s}{2}}\big(W_{v}\rho_{v^{-1}u}\big)(x,c,\eta)\\ =& \int_{v\in \GG^{r(u)}} a_{v}(x,c)(1+e^{2c}\Delta(v^{-1})^{2}(\Delta(v)\eta+\partial\log\Delta(v))^{2})^{-\frac{s}{2}}(W_{v}\rho_{v^{-1}u})(x,c,\eta)\\ =& \int_{v\in \GG^{r(u)}} a_{v}(x,c)(1+e^{2c}(\eta-\partial\log\Delta(v^{-1}))^{2})^{-\frac{s}{2}}(W_{v}\rho_{v^{-1}u})(x,c,\eta),
\end{align*}
where on the last line we have used Lemma \ref{identities} in simplifying
$\Delta(v^{-1})\partial\log\Delta(v) = -\partial\log\Delta(v^{-1})$.  So $a(1+\BB^{2})^{-\frac{s}{2}}$ is the half-density on $H^{*}\rtimes \GG$ defined by
\[
	((x,c,\eta),u)\mapsto 
	a_{u}(x,c)\big(1+e^{2c}(\eta-\partial\log\Delta(u^{-1}))^{2}\big)^{-\frac{s}{2}},
\]
compactly supported in the $u$ and $(x,c)$ variables.  Thus
\begin{align*}
	\tau_{H^{*}}(a(1+\BB^{2})^{-\frac{s}{2}}) =& \int_{M\times\RB\times\RB}a(x,c)\big(1+e^{2c}\eta^{2}\big)^{-\frac{s}{2}}\omega\wedge dc\wedge d\eta\\ =& \int_{Q}a(x,c)d\nu_{Q}\int_{\RB}\big(1+t^{2}\big)^{-\frac{s}{2}}dt,
\end{align*}
where we have made the substitution $t = e^{c}\eta$.  It is then clear that $\tau_{H^{*}}(a(1+\BB^{2})^{-\frac{s}{2}})$ is finite for all $s>1$.
For smoothness, we fix 
$a\in \Gamma_{c}(Q\rtimes \GG;\Omega^{\frac{1}{2}})$ and calculate
\begin{align*}
([\BB^{2},a]&\rho)_{u}(x,c,\eta) 
= e^{2c}\eta^{2}\int_{v\in \GG^{r(u)}} a_{v}(x,c)(W_{v}\rho_{v^{-1}u})(x,c,\eta)\\
&\qquad\qquad\qquad\qquad\qquad- \int_{v\in \GG^{r(u)}} a_{v}(x,c)(W_{v}\BB^{2}_{s(v)}\rho_{v^{-1}u})(x,c,\eta)\\ 
=& \int_{v\in \GG^{r(u)}} a_{v}(x,c)e^{2c}(\eta^{2}-\Delta(v^{-1})^{2}(\Delta(v)\eta+\partial\log\Delta(v))^{2})(W_{v}\rho_{v^{-1}u})(x,c,\eta)\\ =& \int_{v\in \GG^{r(u)}}a_{v}(x,c)e^{2c}\big(2\eta\partial\log\Delta(v^{-1})-(\partial\log\Delta(v^{-1}))^{2}\big)(W_{v}\rho_{v^{-1}u})(x,c,\eta)
\end{align*}
so that $[\BB^{2},a]$ is convolution by the half-density on $H^{*}\rtimes \GG$ defined by
\[
((x,c,\eta),u)\mapsto a_{u}(x,c)e^{2c}\big(2\eta\partial\log\Delta(u^{-1})-(\partial\log\Delta(u^{-1}))^{2}\big).
\]
We also calculate
\begin{align*}
([\BB^{2},[\BB,a]]\rho)_{u}(x,c,\eta) 
=& e^{2c}\eta^{2}\big([\BB,a]\rho\big)_{u}(x,c,\eta)-\big([\BB,a]\BB^{2}\rho\big)_{u}(x,c,\eta)\\ 
=& e^{2c}\eta^{2}\int_{v\in \GG^{r(u)}} a_{v}(x,c)e^{c}\partial\log\Delta(v^{-1})\big(W_{v}\rho_{v^{-1}u}\big)(x,c,\eta)\\ 
&- \int_{v\in \GG^{r(u)}} a_{v}(x,c)e^{c}\partial\log\Delta(v^{-1})\big(W_{v}\BB^{2}_{s(v)}\rho_{v^{-1}u}\big)(x,c,\eta)\\ 
=& e^{2c}\eta^{2}\int_{v\in \GG^{r(u)}} a_{v}(x,c)e^{c}\partial\log\Delta(v^{-1})(W_{v}\rho_{v^{-1}u})(x,c,\eta)\\
&- \int_{v\in \GG^{r(u)}} a_{v}(x,c)e^{3c}\partial\log\Delta(v^{-1})\Delta(v^{-1})^{2}(\Delta(v)\eta+\partial\log\Delta(v))^{2}\\&\times(W_{v}\rho_{v^{-1}u})(x,c,\eta)\\ 
=& \int_{v\in \GG^{r(u)}} a_{v}(x,c)e^{3c}\big(2\eta\partial\log\Delta(v^{-1})-(\partial\log\Delta(v^{-1}))^{2}\big)\\ &\times\partial\log\Delta(v^{-1})(W_{v}\rho_{v^{-1}u})(x,c,\eta),
\end{align*}
so that $[\BB^{2},[\BB,a]]$ is the half-density on $H^{*}\rtimes \GG$ defined by
\[
	((x,c,\eta),u)\mapsto a_{u}(x,c)e^{3c}\big(2\eta\partial\log\Delta(u^{-1})-(\partial\log\Delta(u^{-1}))^{2}\big)\partial\log\Delta(u^{-1}).
\]
More generally, setting $T^{(0)}:=T$ and then inductively defining $T^{(k)}:=[\BB^{2},T^{(k-1)}]$, we see that $[\BB,a]^{(k)}$ is the half-density on $H^{*}\rtimes \GG$ defined by
\[
((x,c,\eta),u)\mapsto a_{u}(x,c)e^{(2k+1)c}\big(2\eta\partial\log\Delta(u^{-1})-(\partial\log\Delta(u^{-1}))^{2}\big)^{k}\partial\log\Delta(u^{-1}).
\]
Now these computations show that for $a\in \AA$ and $k\in\NB$, the operators $a^{(k)}$ and $[\BB,a]^{(k)}$ are half densities on $H^{*}\rtimes \GG$, with compact support in the $((x,c),u)\in Q\rtimes \GG$ variables equal to that of $a$, and growing like $\eta^{k}$ in the fibre variable $\eta\in H^{*}_{(x,c)}$ for all $(x,c)\in Q$. 
Hence both $a^{(k)}(1+\BB^2)^{-k/2}$ and $[\BB,a]^{(k)}(1+\BB^2)^{-k/2}$
are bounded with compact support in the $Q\rtimes \GG$ directions.
Hence for all $a\in \AA$ the operator
$$
(1+\BB^2)^{-k/2-s/4}(a^{(k)})^*a^{(k)}(1+\BB^2)^{-k/2-s/4}
$$
is trace class whenever the real part of $s$ is greater than 1, and similarly
with $a$ replaced by $[\BB,a]$. Thus $\AA\cup [\BB,\AA]\subset B_2^\infty(\BB,1)$
in the notation of \cite{cgrs2}. Thus $\AA^2$, the span of products from $\AA$,
satisfies $\AA^2\cup[\BB,\AA^2]\subset B_1^\infty(\BB,1)$, showing that
the semifinite spectral triple over $\AA^2$ is smoothly summable.

The last step to establish smooth summability is to observe that
$\AA$ has a (left) approximate unit for the inductive limit topology
by \cite[Proposition 6.8]{muhwil}. This ensures that any compactly supported
section in $\AA=\Gamma_{c}(Q\rtimes \GG;\Omega^{\frac{1}{2}})$
can be approximated by products while preserving summability.

Finally the computations also show that $(\AA,\HH,\BB)$
has isolated spectral dimension, as in \cite[Definition 3.1]{cgrs2},
since for all multi-indices $k$ of length $m\geq 0$ we have proved that
$$
\tau_{H^*}(a_0[\BB,a_1]^{(k_1)}\cdots[\BB,a_m]^{(k_m)}(1+\BB^2)^{-|k|-m/2-s})
$$
has a meromorphic continuation in a neighbourhood of $s=0$.
\end{proof}

Finally we can prove the Theorem \ref{mainthm}.

\begin{proof}[Proof of Theorem \ref{mainthm}]
Since the spectral dimension $p = 1$ and since the 
parity of the spectral triple is 1, the only nonzero 
term in the residue cocycle is $\phi_{1}$ as defined in \cite[Definition 3.2]{cgrs2}.  For any $a\in \Gamma_{c}(Q\rtimes \GG;\Omega^{\frac{1}{2}})$ we have
\begin{align*}
	\big([\BB,a]\rho\big)_{u}(x,c,\eta) 
	=& \BB_{r(u)}\int_{v\in \GG^{r(u)}} a_{v}(x,c)(W_{v}\rho_{v^{-1}u})(x,c,\eta)\\&-\int_{v\in \GG^{r(u)}} a_{v}(x,c)(W_{v}\BB_{s(v)}\rho_{v^{-1}u})(x,c,\eta)\\ 
	=& \int_{v\in \GG^{r(u)}} a_{v}(x,c)\big(\BB_{r(v)}- W_{v}\BB_{s(v)} W_{v^{-1}}\big)(W_{v}\rho_{v^{-1}u})(x,c,\eta)\\ 
	=& \int_{v\in \GG^{r(u)}} a_{v}(x,c)e^{c}\partial\log\Delta(v^{-1})(W_{v}\rho_{v^{-1}u})(x,c,\eta)\\ =& (\delta_{1}(a)\rho)_{u}(x,c,\eta),
\end{align*}
where $\delta_{1}$ is the derivation of $\Gamma_{c}(Q\rtimes \GG;\Omega^{\frac{1}{2}})$ defined by
\[
	\delta_{1}(a)_{u}(x,c):=e^{c}\partial\log\Delta(u^{-1})a_{u}(x,c).
\]
The derivation $\delta_{1}$ is the non-\'{e}tale analogue of that given in \cite[Page 39]{backindgeom}.  
Thus for $a_{0},a_{1}\in \Gamma_{c}(Q\rtimes \GG;\Omega^{\frac{1}{2}})$, we calculate
\begin{align*}
	\phi_{1}(a_{0},a_{1}) 
	=& 2(2\pi i)^{\frac{1}{2}}\res_{z=0}\tau_{H^{*}}\big(a_{0}[\BB,a_{1}](1+\BB^{2})^{-\frac{1}{2}-z}\big)\\ 
	=& 2(2\pi i)^{\frac{1}{2}}\tau_{Q}(a_{0}\delta_{1}(a_{1}))\res_{z=0}\int_{\RB}(1+t^2)^{-\frac{1}{2}-z}dh\\ 
	=&2(2\pi i)^{\frac{1}{2}}\tau_{Q}(a_{0}\delta_{1}(a_{1}))\res_{z=0}\frac{\Gamma(1/2)\Gamma(z)}{2\Gamma(1/2+z)}\\
	=&(2\pi i)^{\frac{1}{2}}\tau_{Q}(a_{0}\delta_{1}(a_{1})).
\end{align*}
This is, up to the factor $(2\pi i)^{\frac{1}{2}}$, the non-\'{e}tale analogue of the Godbillon-Vey cyclic cocycle from
\cite[Proposition 19]{backindgeom}.
\end{proof}

\section{Relation with Connes' approach}\label{relconnes}

We will outline in this section how our construction can, in codimension 1, be reconciled with Connes' approach to realise Gelfand-Fuks cocycles as cyclic cocycles for a convolution algebra.  In doing so we will be able to justify why our spectral triple represents the Godbillon-Vey invariant.  Let us first briefly recall Connes' approach \cite[Theorem 7.15]{cyctrans}.

\subsection{Connes' approach}

Connes considers a discrete group $\Gamma$ of orientation preserving diffeomorphisms of an oriented manifold $V$ of dimension $n$.  Associated to $V$ and any $k\in\NB\cup\{0\}$ is its \emph{positively oriented $k^{th}$ order jet bundle} $J^{+}_{k}(V)$, whose fibre over $x\in V$ consists of equivalence classes of local diffeomorphisms $\varphi:U\rightarrow V$, where $U$ is an open neighbourhood of $0$ in $\RB^{n}$, for which $\varphi(0) = x$.  Two such diffeomorphisms $\varphi$ and $\psi$ are said to have the same \emph{$k$-jet at $0$}, denoted $j^{k}_{0}(\varphi) = j^{k}_{0}(\psi)$, if in any local coordinate system about $x$ all partial derivatives of $\varphi$ and $\psi$ of order less than or equal to $k$ coincide at $0$.  All the bundles $J^{+}_{k}(V)$, for $k\geq1$, carry canonical right actions of $GL^{+}(n,\RB)$, and we write $\underline{J}^{+}_{k}(V):=J^{+}_{k}(V)/SO(n,\RB)$.  The $\underline{J}^{+}_{k}(V)$ have contractible fibres \cite[p. 132]{bott3}. 

It is then well-known \cite[Equation 3.2]{bott3} that associated to any ($SO(n,\RB)$-relative) Gelfand-Fuchs cocycle $\varpi$ is $k\in\NB$ and a canonical, diffeomorphism-invariant differential form, also denoted $\varpi$, on the quotient $\underline{J}^{+}_{k}(V)$. Connes uses the exterior derivative on $V$ to manufacture $\varpi$ into a cyclic cocycle $\varpi_{c}$ for the algebra $C_{c}^{\infty}(\underline{J}^{+}_{k}(V)\rtimes\Gamma)$. Letting $W$ denote the bundle of metrics over $V$, Connes invokes work of Kasparov \cite{kasconsp} to deduce the existence of a class $j_{1,k}\in KK^{\Gamma}(C_{0}(W),C_{0}(\underline{J}^{+}_{k}(V)))$, whose descent $j_{1,k}^{\Gamma}\in KK(C_{0}(W\rtimes\Gamma),C_{0}(\underline{J}^{+}_{k}(V)\rtimes\Gamma))$ implements an isomorphism on $K$-theory via the Kasparov product.  Letting $j_{0,1}^{\Gamma}\in KK(C_{0}(V\rtimes\Gamma),C_{0}(W\rtimes\Gamma))$ denote the bundle of metrics Kasparov module constructed in \cite[Section 5]{cyctrans}, 
Connes obtains a linear map
\[
K_{*}(C_{0}(V)\rtimes\Gamma)\ni a\mapsto \varpi_{c}(a\otimes_{C_{0}(V)\rtimes\Gamma} j_{0,1}^{\Gamma}\otimes_{C_{0}(W)\rtimes\Gamma} j_{1,k}^{\Gamma})\in\RB
\]
defined by $\varpi$.  In the case when $V = S^{1}$, one has $J^{+}_{2}(S^{1}) = S^{1}\times\RB^{*}_{+}\times\RB$, and if $dx$ is the standard volume form on $S^{1}$ then the Godbillon-Vey invariant is represented by the invariant differential form \cite[Lemma 7.7]{cyctrans}
\begin{equation}\label{connesgv}
\varpi = \frac{1}{y^{3}}dx\wedge dy\wedge dy_{1},\hspace{7mm}(x,y,y_{1})\in S^{1}\times\RB^{*}_{+}\times\RB.
\end{equation}
The associated cyclic cocycle $\varpi_{c}$ is the trace on $C_{c}^{\infty}(J^{+}_{2}(S^{1}))\rtimes\Gamma$ obtained by integration with respect to $\varpi$, and an involved calculation \cite[Theorem 7.3]{cyctrans} shows that the linear map thus obtained on $K_{0}(C_{0}(V)\rtimes\Gamma)$ is the cyclic cocycle given by Equation \eqref{cgv}.  Alternatively, one can obtain by essentially the same method a map
\begin{equation}\label{conneswgv}
K_{1}(C_{0}(W)\rtimes\Gamma)\ni a\mapsto\varpi_{c}(a\otimes_{C_{0}(W)\rtimes\Gamma}j^{\Gamma}_{1,2})\in\RB
\end{equation}
on the $K$-theory of $C_{0}(W)\rtimes\Gamma$.  In the next subsection we will indicate how the index pairing induced by the spectral triple in Theorem \ref{mainthm} can be thought of as a non-\'{e}tale version of the map in Equation \eqref{conneswgv}.

\subsection{The case of a general foliation}

In the setting of a transversely orientable foliated manifold $(M,\FF)$ of codimension $q$, one has access to the \emph{transverse jet bundles} $J^{+}_{k}(\FF)$ \cite[p. 113 - 114]{gcc}.  The fibre $J^{+}_{k}(\FF)_{x}$ over $x\in M$ now consists of the $k$-jets $j^{k}_{0}(\varphi)$ of orientation-preserving local diffeomorphisms $\varphi$ sending an open neighbourhood of $0\in\RB^{q}$ to a local \emph{transversal} about $x = \varphi(0)$.  Note in particular that $J^{+}_{1}(\FF)$ is the same thing as the oriented transverse frame bundle $F^{+}(N)$.

We have natural projections $\pi_{k}:J^{+}_{k}(\FF)\rightarrow M$ and $\pi_{k+1,k}:J^{+}_{k+1}(\FF)\rightarrow J^{+}_{k}(\FF)$ defined respectively by forgetting all partial derivatives, and by forgetting all partial derivatives of order $k+1$.  Moreover the holonomy groupoid $\GG$ acts naturally on each $J^{+}_{k}(\FF)$, and its orbits define a foliation $\FF_{k}$ of $J^{+}_{k}(\FF)$.  There is a canonical right action of $GL^{+}(q,\RB)$ on all of the $J^{+}_{k}(\FF)$, $k\geq 1$, which commutes with the action of $\GG$.

\begin{prop}\label{identify}
	Let $S^{2}(\RB^{q})$ denote the symmetric polynomials of degree 2 over the vector space $\RB^{q}$.  Then the bundle $J^{+}_{2}(\FF)$ is an affine bundle over $J^{+}_{1}(\FF)$, modelled on the vector bundle $\pi_{1}^{*}(N)\otimes S^{2}(\RB^{q})$.  Moreover, if $(M,\FF)$ is of codimension 1, a choice of Bott connection $\nabla^{\flat}$ on $N$ determines an affine isomorphism of $J^{+}_{2}(\FF)$ with the total space of the bundle $H^{*} = \big(\ker(\alpha^{\flat}/T\FF_{1})\big)^{*}$ over $J^{+}_{1}(\FF)$.
\end{prop}

\begin{proof}
	That $J^{+}_{2}(\FF)$ is an affine bundle over $J^{+}_{1}(\FF)$ modelled on $\pi_{1}^{*}(N)\otimes S^{2}(\RB^{q})$ can be seen using local coordinates.  Indeed, if $U\subset \RB^{q}$ is an open neighbourhood of $0$ and $T_{\beta}$ is a local transversal in $M$, then the 2-jet at $0$ of any diffeomorphism $\varphi:U\rightarrow T$ is distinguished from its 1-jet at $0$ by the partial derivatives
	\[
	\frac{\partial^{2}\varphi^{i}}{\partial y^{j}\partial y^{k}}\bigg|_{0},\hspace{7mm}i,j,k = 1,\dots,q
	\]
	which are elements of $N_{x}\otimes S^{2}(\RB^{q})$ (here we have identified $\RB^{q}$ with its dual in the natural way).  The chain rule implies that whenever $c_{\alpha\beta}:T_{\beta}\rightarrow T_{\alpha}$ is the diffeomorphism defined by a transverse change of coordinates, then the corresponding coordinate change in the fibre $J^{+}_{2}(\FF)_{j^{1}_{0}(\varphi)}$ is given by the affine transformation
	\[
	\bigg(\frac{\partial^{2}\varphi^{i}}{\partial y^{j}\partial y^{k}}\bigg|_{0}\bigg)_{i,j,k = 1,\dots,q}\mapsto \bigg(\frac{\partial c^{i}_{\alpha\beta}}{\partial y^{l}}\bigg|_{\varphi(0)}\frac{\partial^{2}\varphi^{l}}{\partial y^{j}\partial y^{k}}\bigg|_{0}+\frac{\partial^{2}c_{\alpha\beta}^{i}}{\partial y^{j}\partial y^{k}}\bigg|_{\varphi(0)}\bigg)_{i,j,k=1,\dots,q},
	\]
	where we have used the Einstein summation convention.
	
	Suppose now that $\nabla$ is a torsion-free connection on $N$.  Then on any sufficiently small transversal $\mathcal{Y}_{x}$ about $x\in M$, $\nabla$ restricts to a torsion-free affine connection and is therefore associated with an exponential map $\exp^{\nabla}_{x}$ which sends an open neighbourhood of $0\in N_{x} = T_{x}\mathcal{Y}_{x}$ diffeomorphically onto $\mathcal{Y}_{x}$.  We therefore obtain a section $\sigma_{\nabla}:J^{+}_{1}(\FF)\rightarrow J^{+}_{2}(\FF)$ defined by
	\begin{equation}\label{sect12}
	\sigma_{\nabla}(j^{1}_{0}(\varphi)):=j^{2}_{0}(\exp^{\nabla}_{x}\circ j^{1}_{0}(\varphi)),
	\end{equation}
	where we consider $j^{1}_{0}(\varphi)$ as a frame $\RB^{q}\rightarrow N_{x}$, which by the chain rule is equivariant under the right action of $GL^{+}(q,\RB)$.  Therefore $\nabla$ determines an affine isomorphism of $J^{+}_{2}(\FF)$ with the vector bundle $\pi_{1}^{*}(N)\otimes S^{2}(\RB^{q})$ on which it is modelled.  In particular, if $(M,\FF)$ is codimension 1, $\nabla$ determines an affine isomorphism of $J^{+}_{2}(\FF)$ with $\pi_{1}^{*}(N)$.  Therefore if $\nabla = \nabla^{\flat}$ is a Bott connection, we attain the stated isomorphism of $J^{+}_{2}(\FF)$ with $H\cong\pi_{1}^{*}(N)$, and hence with $H^{*}\cong H$ via the tautological Euclidean structure of Proposition \ref{bundleofmetrics}.
\end{proof}

Now any Gelfand-Fuchs class $\varpi$ admits a canonical representative in the $\GG$-invariant forms on $\underline{J}^{+}_{k}(\FF)$ for sufficiently large $k$ \cite[p. 117 - 119]{gcc}.  Let us briefly describe this canonical representative in the case of the Godbillon-Vey invariant of a codimension 1 foliation.  Let $\varphi_{t}:U\rightarrow\mathcal{Y}$ be a 1-parameter family of diffeomorphisms of some open neighbourhood $U$ of $0\in\RB$ onto a local transversal $\mathcal{Y}$ in $M$, determining tangent vectors
\[
X_{k}:=\frac{d}{dt}\bigg|_{t=0}j^{k+1}_{0}(\varphi_{t})
\]
in $J^{+}_{k+1}(\FF)$, for all $k\in\NB\cup\{0\}$.  Let $\varphi:=\varphi_{0}$, and define a 1-parameter family of coordinates $x_{t}:=\varphi^{-1}\circ\varphi_{t}$.  Then we define the \emph{$k^{th}$ tautological 1-form} $\omega^{k}$ on $J^{+}_{k+1}(\FF)$ by the formula
\begin{equation}\label{omegak}
\omega^{k}(X_{k}) = \frac{d}{dt}\bigg|_{t=0}\bigg(\frac{d^{k} x_{t}}{d y^{k}}\bigg|_{y=0}\bigg)
\end{equation}
for all $k\in\NB\cup\{0\}$.  Such tautological forms were introduced (in the non-foliated case) by Kobayashi \cite{kob}.  Notice that if $u\in \GG$  is represented by a diffeomorphism $h:\mathcal{Y}\rightarrow\mathcal{Y}'$ of local transversals, then
\[
(h^{*}\omega^{k})(X_{k}) = \omega^{k}\bigg(\frac{d}{dt}\bigg|_{t=0}j^{k+1}_{0}(h\circ\varphi_{t})\bigg) = \frac{d}{dt}\bigg|_{t=0}\bigg(\frac{d^{k}(\varphi^{-1}\circ h^{-1}\circ h\circ \varphi_{t})}{dy^{k}}\bigg) = \omega^{k}(X_{k}),
\]
so the $\omega^{k}$ are $\GG$-invariant.  Pulling back via the projections, let us assume that $\omega^{0}$, $\omega^{1}$, and $\omega^{2}$ are all defined on $J^{+}_{3}(\FF)$.  Then, using the $\GG$ invariance of these forms together with \cite[Proposition 3.23]{gcc}, it can be shown that the $\omega^{i}$ satisfy the structure equations
\begin{equation}\label{struct}
d\omega^{0} = -\omega^{1}\wedge\omega^{0},\hspace{7mm} d\omega^{1} = -\omega^{2}\wedge\omega^{0}.
\end{equation}
Now by Equation \eqref{omegak} the form $\omega^{0}$ on $J^{+}_{1}(\FF)$ is, by definition, simply the solder form on the transverse frame bundle.  Therefore the first of the Equations \eqref{struct} says that $\omega^{1}$ on $J^{+}_{2}(\FF)$ behaves like a torsion-free connection form: indeed, the pullback $\sigma_{\nabla}^{*}\omega^{1}$ of $\omega^{1}$ by the section $\sigma_{\nabla}:J^{+}_{1}(\FF)\rightarrow J^{+}_{2}(\FF)$ defined by any torsion-free connection $\nabla$ on $N$ is precisely the connection form corresponding to $\nabla$.  Therefore the second of the Equations \eqref{struct} is a formula for the ``curvature" of the ``connection form" $\omega^{1}$.  In particular, the Godbillon-Vey form on $J^{+}_{3}(\FF)$ is simply the (negative of) the ``connection" wedged with its ``curvature" \cite[Lemma 10.9]{bott2}:
\begin{equation}\label{gvtaut}
\gv = -\omega^{1}\wedge d\omega^{1} = \omega^{0}\wedge\omega^{1}\wedge\omega^{2}.
\end{equation}
It is by no means obvious that the form in Equation \eqref{gvtaut} is related to the form $d\nu_{H^{*}}$ on $H^{*}$ considered in Equation \eqref{volforms}.  The next result tells us that in fact these forms are the same, and therefore justifies our claim that the cocycle obtained as the index formula from Theorem \ref{mainthm} truly does represent the Godbillon-Vey invariant.

\begin{thm}
	Let $(M,\FF)$ be a transversely oriented foliation of codimension 1, with transverse volume form $\omega\in\Omega^{1}(M)$.  Suppose moreover we are gifted with a torsion-free Bott connection on $N$, giving an identification $J^{+}_{2}(\FF)\cong H^{*}$ as in Proposition \ref{identify}.  Then the form $\gv$ of Equation \eqref{gvtaut} on $J^{+}_{3}(\FF)$ descends to a form on $J^{+}_{2}(\FF)\cong H^{*}$, which, in the trivialisation $H^{*} = M\times\RB\times\RB$ determined by $\omega$ as in Subsection \ref{subsec:simp}, coincides the the form $d\nu_{H^{*}} = \omega\wedge dc\wedge d\eta$ of Equation \eqref{volforms}.
\end{thm}

\begin{proof}
	Associated to the transverse volume form $\omega$ is a nonvanishing normal vector field $Z\in\Gamma^{\infty}(M;N)$ characterised by the equation $\omega(Z)\equiv 1$.  Fix $x\in M$ and let $\mathcal{Y}_{x}$ be a local transversal through $x$.  Then the torsion-free Bott connection $\nabla^{\flat}$ on $N$ restricts to an affine connection on $\mathcal{Y}_{x}$, and so determines an exponential map $\exp^{\nabla^{\flat}}:U\rightarrow\mathcal{Y}_{x}$ which is a local diffeomorphism defined on an open neighbourhood $U$ of $0\in T_{x}\mathcal{Y}_{x}$.  Rescaling $\omega$ if necessary, we can always assume that $Z_{x}\in U$ and we obtain a coordinate $u_{0}:\mathcal{Y}_{x}\rightarrow\RB$ defined by the equation
	\[
	u_{0}(x')Z_{x} = \big(\exp^{\nabla^{\flat}}\big)^{-1}(x'), \hspace{7mm}x'\in\mathcal{Y}_{x}.
	\]
	Now fix a local diffeomorphism $\varphi$ from an open neighbourhood of $0\in\RB$ to $\mathcal{Y}_{x}$.  The coordinate $u_{0}$ on $\mathcal{Y}_{x}$ identifies $\varphi$ with a local diffeomorphism $\tilde{\varphi}:=u_{0}\circ\varphi$ of $\RB$, so the 3-jet $j^{3}_{0}(\varphi)$ is determined by the polynomial
	\[
	\tilde{\varphi}(0)+\frac{d\tilde{\varphi}}{dy}\bigg|_{0}y+\frac{d^{2}\tilde{\varphi}}{dy^{2}}\bigg|_{0}y^{2} + \frac{d^{3}\tilde{\varphi}}{dy^{3}}\bigg|_{0}y^{3}
	\]
	where we use $y$ to denote the standard coordinate in $\RB$.  We thus define coordinates $u_{i}(j^{3}_{0}(\varphi)):=\frac{d^{i}\tilde{\varphi}}{dy^{i}}\big|_{0}$ for $i = 1,2,3$ for $j^{3}_{0}(\varphi)\in J^{+}_{3}(\mathcal{Y}_{x})$.  Suppose now that $\varphi_{t}$ is a 1-parameter family of local diffeomorphisms from an open neighbourhood of $0\in\RB$ to $\mathcal{Y}_{x}$, with $\varphi_{0} = \varphi$.  The coordinate representation $\tilde{\varphi}_{t}:=u_{0}\circ\varphi_{t}$of the family $\varphi_{t}$ determines a curve
	\[
	j^{3}_{0}(\tilde{\varphi}_{t}) = \bigg(\tilde{\varphi}_{t}(0),\frac{d\tilde{\varphi}_{t}}{dx}\bigg|_{0},\frac{d^{2}\tilde{\varphi}_{t}}{dx^{2}}\bigg|_{0}\frac{d^{3}\tilde{\varphi}_{t}}{dx^{3}}\bigg|_{0}\bigg)
	\]
	in $J^{+}_{3}(\RB)$, hence we can write the tangent vector $X = \frac{d}{dt}\big|_{0}j^{3}_{0}(\varphi_{t})$ on $J^{+}_{3}(\mathcal{Y}_{x})$ determined by the curve $j^{3}_{0}(\varphi_{t})$ in the form
	\[
	X = \frac{d\tilde{\varphi}_{t}}{dt}\bigg|_{0}\frac{\partial}{\partial u^{0}}+\frac{d}{dt}\bigg|_{0}\bigg(\frac{d\tilde{\varphi}_{t}}{dy}\bigg|_{0}\bigg)\frac{\partial}{\partial u^{1}}+\frac{d}{dt}\bigg|_{0}\bigg(\frac{d^{2}\tilde{\varphi}_{t}}{dy^{2}}\bigg|_{0}\bigg)\frac{\partial}{\partial u^{2}}+\frac{d}{dt}\bigg|_{0}\bigg(\frac{d^{3}\tilde{\varphi}_{t}}{dy^{3}}\bigg|_{0}\bigg)\frac{\partial}{\partial u^{3}}.
	\]
	Setting $h_{t}:=\varphi^{-1}\circ\varphi_{t}$ we have $\tilde{\varphi}_{t} = \tilde{\varphi}\circ h_{t}$.  Then using the chain rule together with the fact that $h_{0} = \id_{\RB}$, we compute
	\[
	\frac{d\tilde{\varphi}_{t}}{dt}\bigg|_{0} = \frac{d}{dt}\bigg|_{0}(\tilde{\varphi}\circ h_{t}) = u_{1}\frac{dh_{t}}{dt}\bigg|_{0},
	\]
	\[
	\frac{d}{dt}\bigg|_{0}\bigg(\frac{d\tilde{\varphi}_{t}}{dy}\bigg|_{0}\bigg) = \frac{d}{dt}\bigg|_{0}\bigg(\frac{d(\tilde{\varphi}\circ h_{t})}{dy}\bigg|_{0}\bigg) = u_{2}\frac{dh_{t}}{dt}\bigg|_{0}+u_{1}\frac{d}{dt}\bigg|_{0}\bigg(\frac{d h_{t}}{dy}\bigg|_{0}\bigg),
	\]
	and
	\[
	\frac{d}{dt}\bigg|_{0}\bigg(\frac{d^{2}\tilde{\varphi}_{t}}{dy^{2}}\bigg|_{0}\bigg) = \frac{d}{dt}\bigg|_{0}\bigg(\frac{d^{2}(\tilde{\varphi}\circ h_{t})}{dy^{2}}\bigg|_{0}\bigg) = u_{3}\frac{dh_{t}}{dt}\bigg|_{0}+2u_{2}\frac{d}{dt}\bigg|_{0}\bigg(\frac{dh_{t}}{dy}\bigg|_{0}\bigg)+u_{1}\frac{d}{dt}\bigg|_{0}\bigg(\frac{d^{2} h_{t}}{dy^{2}}\bigg|_{0}\bigg).
	\]
	Therefore by Equation \eqref{omegak} we find that\footnote{The formulae in Equation \eqref{dodgys} that we compute here differ slightly from the analogous equations of Kobayashi \cite[Section 4]{kob} and Connes-Moscovici \cite[p. 45]{backindgeom}, for whom the summands containing $\omega^{0}$ in the second and third equations have factors of 2 and 3 respectively.  The reader can easily verify using elementary calculus that our own computations do \emph{not} give rise to these factors.  In the absence of any explicit computations provided by Kobayashi and Connes-Moscovici, it is difficult to determine why these additional factors appear in their equations.  In any case, these additional factors have no impact on the coordinate expression we obtain for the Godbillon-Vey differential form.}
	\begin{equation}\label{dodgys}
	du_{0} = u_{1}\omega^{0},\hspace{5mm}du_{1} = u_{2}\omega^{0}+u_{1}\omega^{1},\hspace{5mm} du_{2} = u_{3}\omega^{0}+2u_{2}\omega^{1}+u_{1}\omega^{2},
	\end{equation}
	and we deduce that
	\begin{equation}\label{omegagv}
	\omega^{0}\wedge\omega^{1}\wedge\omega^{2} = \frac{1}{u_{1}^{3}}du_{0}\wedge du_{1}\wedge du_{2},
	\end{equation}
	which is a well-defined form on $J^{+}_{2}(\mathcal{Y}_{x})$.
	
	Now we come to transporting the form of Equation \eqref{omegagv} on $J^{+}_{2}(\mathcal{Y}_{x})\subset J^{+}_{2}(\FF)$ to the total space of the bundle $H^{*}$ over $J^{+}_{1}(\FF)$ as in Proposition \ref{identify}.  The transverse volume form $\omega\in\Omega^{1}(M)$ determines a trivialisation
	\[
	J^{+}_{1}(\FF)\ni\phi_{x}\mapsto(x,\omega_{x}(\phi_{x}(1)))=:(x,t)\in M\times\RB^{*}_{+},
	\]
	where we think of $\phi_{x} = d\varphi_{0}$ as a frame $\RB\rightarrow N_{x}$.  The transverse vector field $Z\in\Gamma^{\infty}(M;N)$ corresponding to $\omega$ determines a trivialisation
	\[
	N\ni hZ_{x}\mapsto (x,h)\in M\times\RB,\hspace{5mm}x\in M.
	\]
	of $N$ and therefore a corresponding trivialisation $H\cong J^{+}_{1}(\FF)\times\RB\cong M\times\RB^{*}_{+}\times\RB$ of $H\cong\pi_{1}^{*}N$.  Unlike the coordinates $u_{i}$ used for the transversal $\mathcal{Y}_{x}$ in the first part of the proof, the trivialisation $H\cong M\times\RB^{*}_{+}\times\RB$ is global, and we must show that on $\mathcal{Y}_{x}$, we have equalities $du_{0} = \omega$, $u_{1} = t$ and $u_{2} = h$.  
	
	Now $du_{0}(Z)\equiv1$ by definition of the coordinate $u_{0}$ on $\mathcal{Y}_{x}$ so $du_{0} = \omega$.  For $u_{1}$ we see that
	\[
	u_{1} = \frac{d(u_{0}\circ\varphi)}{dy} = du_{0}\circ d\varphi(1) = \omega_{x}(d\varphi(1)) = t
	\]
	by definition of the variable $t$.  Finally, in the trivial bundle $J^{+}_{2}(\RB) = \RB\times\RB^{*}_{+}\times\RB$ that is the image of $J^{+}_{2}(\mathcal{Y}_{x})$ under the coordinates $(u_{0},u_{1},u_{2})$, the $u_{2}$ variable identifies with the tangent variable for $\RB$ in the manner of Proposition \ref{identify}.  Viewed as coordinates on $J^{+}_{2}(\mathcal{Y}_{x})$ and $T\mathcal{Y}_{x}$ respectively, we then have $u_{2} = h$ and therefore
	\[
	\gv = \frac{1}{u_{1}^{3}}\,du_{0}\wedge du_{1}\wedge du_{2} = -\frac{1}{t^{3}}\,\omega\wedge dt\wedge dh.
	\]
	Finally, as in Subsection \ref{subsec:simp}, we make the substitutions $t = e^{c}$ giving $dt = e^{c}dc$, and $h = e^{2c}\eta$ giving $dh = e^{2c}d\eta$.  Thus, on $H^{*}$, we find that
	\[
	\gv = \omega\wedge dc\wedge d\eta = d\nu_{H^{*}}
	\]
	as claimed.
\end{proof}

\section{Concluding remarks}

It is tempting to view the higher codimension version 
of the codimension 1 
Kasparov
module and spectral triple as analogous data 
representing the Godbillon-Vey invariant in 
higher codimension. Sadly, despite the naturality of the constructions
presented here, it is far from clear that such an interpretation is warranted.
Without an identification of the Chern character of these spectral triples
with the Godbillon-Vey class, they must remain an interesting construction.

One final remark on the constructions presented here: they all pass to
real algebras and real $KK$-theory. All our constructions are Real
\cite{KasparovTech} for the
obvious variations of complex conjugation, in part because of our systematic use
of the exterior algebra rather than the spinor bundle. This means that
we can at all stages retain contact with homology of manifolds with real coefficients.

\bibliographystyle{amsplain}
\bibliography{references}

\end{document}